\documentclass[11pt,a4paper,reqno]{amsart}
\usepackage{amsmath,amssymb,amsthm,amsfonts}
\usepackage[export]{adjustbox}
\usepackage{amsbsy}
\usepackage[utf8]{inputenc}
\usepackage[normalem]{ulem}
\usepackage{xspace}
\usepackage{cancel}
\usepackage[greek,english]{babel}
\usepackage{url}
\usepackage{wrapfig}
\usepackage{enumitem}
\usepackage{multicol}
\usepackage{moreenum}
\usepackage{xcolor}
\usepackage{longtable}
\usepackage{array}

\usepackage[pagewise]{lineno} 
\newcounter{dummy}
\makeatletter
\newcommand\myitem[1][]{\item[#1]\refstepcounter{dummy}\def\@currentlabel{#1}}
\makeatother
\usepackage{subcaption}
\usepackage{comment}
\usepackage{tikz}
\usetikzlibrary{backgrounds}
\usepackage[bottom=2.5cm,top=2.5cm, left=2.5cm, right=2.5cm]{geometry}

\usepackage{multirow}
\usepackage{array,multirow}

\definecolor{LinkColor}{rgb}{0,0,0} 
\usepackage[colorlinks=true,linkcolor=LinkColor,citecolor=LinkColor,urlcolor= LinkColor, naturalnames, hyperindex, pdfstartview=FitH, bookmarksnumbered, plainpages]{hyperref}
\usepackage{cleveref}

\makeatletter
\newcommand{\slunlhd}{%
	\mathrel{\mathpalette\sl@unlhd\relax}%
}

\newcommand{\sl@unlhd}[2]{%
	\sbox\z@{$#1\lhd$}%
	\sbox\tw@{$#1\leqslant$}%
	\dimen@=\ht\tw@
	\advance\dimen@-\ht\z@
	\ifx#1\displaystyle
	\advance\dimen@ .2pt
	\else
	\ifx#1\textstyle
	\advance\dimen@ .2pt
	\fi
	\fi
	\ooalign{\raisebox{\dimen@}{$\m@th#1\lhd$}\cr$\m@th#1\leqslant$\cr}%
}
\makeatother

\newtheorem{innercustomthm}{Theorem}[section]

\crefname{innercustomthm}{Theorem}{Theorems}

\newtheorem{theorem}{Theorem}[section]
\newtheorem{corollary}[theorem]{Corollary}
\newtheorem{lemma}[theorem]{Lemma}
\newtheorem{proposition}[theorem]{Proposition}

\theoremstyle{definition}

\newtheorem{remark}[theorem]{Remark}
\newtheorem{example}[theorem]{Example}

\newcommand{\Irr}{\operatorname{Irr}}

\newcommand{\F}{\textup{F}}
 
\newcommand{\tr}{\textup{tr}}

\author[S. Chahal]{Seema Chahal}
\address{(Seema Chahal) Department of Mathematics, Indian Institute of Technology Roorkee, Roorkee (Uttarakhand)-247667, India.}
\email{\href{mailto:seema_r@ma.iitr.ac.in}{seema\_r@ma.iitr.ac.in}}
\author[S. Maheshwary]{Sugandha Maheshwary}
\address{(Sugandha Maheshwary) Department of Mathematics, Indian Institute of Technology Roorkee, Roorkee (Uttarakhand)-247667, India.}
\email{\href{mailto:msugandha@ma.iitr.ac.in}{msugandha@ma.iitr.ac.in}}

\thanks{The second author gratefully acknowledges the support by Science  \& Engineering Research Board (SERB),  DST (Department of Science and Technology), India (SRG/2023/000180).}

\keywords{finite semisimple group algebras, primitive central idempotents, group codes.}

\subjclass[2010]{16S34, 20C05, 11T71, 94B05, 94B65}
\date{}

\title{On metacyclic $p$-group codes }
\begin{document}
	\maketitle
	\begin{abstract}
		In this article, we study the metacyclic $p$-group codes arising from finite semisimple group algebras. In \cite{CM25}, we studied group codes arising from metacyclic groups with order divisible by two distinct odd primes. In the current work, we focus on metacyclic $p$-group codes, as a result of which we are also able to extend the results of \cite{CM25} for metacyclic 	groups with order divisible by any two primes, not necessarily odd or distinct. Consequently, existing results on group algebras of some important classes of groups, including dihedral and quaternion groups, have been extended. Additionally, we provide left codes for the undertaken group algebras. 
		 Finally, we construct non-central codes using units motivated by  Bass and bicyclic units, which  are inequivalent to any abelian group codes and yield  best known parameters.
	\end{abstract}
	
	\section{Introduction}
	The theory of group codes, initiated by Berman \cite{Ber67} and MacWilliams~\cite{Wil70}, studies ideals of semisimple group algebras. Since these ideals are determined by idempotents, their explicit description plays a central role in the construction and analysis of group codes. In particular, primitive central idempotents (pcis) yield information about central codes. Group codes form a rich class of linear codes.  For instance, cyclic codes can be understood as group codes arising from cyclic groups, Reed-Solomon codes over field $\mathbb{F}_p$ are group codes of elementary abelian $p$-groups \cite{Pas88}; the binary Golay code $[24,12,8]$ can be obtained as an ideal in a group algebra over a finite field~\cite{Wol80}.  
	
	In group codes, abelian group codes are well studied and cover many classical families of linear codes. However non-abelian codes are also of interest because of their potential applications in code-based cryptography (\cite{DK15}, \cite{DK16}). Among the non-abelian group codes, metacyclic codes form an asymptotically good family of codes \cite{BMS20}. Particularly for dihedral  codes over $\mathbb{F}_q$, Dutra et al.~\cite{DFPM09} investigated the codes under certain restrictions on $q$. Under similar restrictions, Assuena and Milies~(\cite{APM17}, \cite{APM19}, \cite{Ass22}) considered split metacyclic groups of order $p_1^m p_2^n$, where $p_1$ and $p_2$ are distinct odd primes. They also proposed constructions of certain non-central codes with good parameters. Gupta and Rani~(\cite{GR22b}, \cite{GR22a}, \cite{GR23}) applied the theory of strong Shoda pairs to obtain pcis for dihedral groups and constructed corresponding codes, again under restrictive hypotheses on $q$. More recently, Vedenev~\cite{Ved25} carried out a comprehensive study of group codes from non-abelian split metacyclic group algebras.  
	
	In our earlier work~\cite{CM25}, we mainly worked with the pcis of $\mathbb{F}_qG$, where $G$ is a metacyclic group of order $p_1^m p_2^l$, with $p_1$ and $ p_2$  distinct odd primes. Unlike the previously existing work as cited above, we assumed almost no restriction on $q$ and hence extended the known results in this direction. In the current article, we further extend the results of \cite{CM25} by including the cases where $p_1$ and $p_2$ are any primes, not necessarily odd or distinct. This is done by studying the pcis in metacyclic $p$-group algebras, for any prime $p$. In particular, for metacyclic $p$-groups which  have a maximal cyclic subgroup, we also obtain the structures for their respective group algebras. Consequently, we improve several existing results on dihedral $2$-codes as well (cf.~\cite{DFPM09}, \cite{GR22b}, \cite{GR22a}). Furthermore, the results are generalised for metacyclic groups with order divisible by more than two primes and dihedral as well as Quaternion groups of arbitrary orders. 
	
	Throughout the article, we consider various kinds of groups. For each of these groups, we compute a complete set of pcis as well as left idempotents and study all the parameters for the associated group codes. We also provide an $\mathbb{F}_q$-basis for these codes. 
	The computation of pcis is based on strong Shoda pair theory and wherever possible, we provide unified treatment to the codes depending upon the type of the corresponding  strong Shoda pairs. This also facilitates us to give the explicit structure of the considered group algebras. Finally, we construct non-central codes via conjugation of idempotents with suitable units, motivated by well known Bass and bicyclic units of $\mathbb{Z}G$. Hence, we obtain non-central codes with improved parameters as compared to central codes.  
	We include illustrations through explicit construction of codes whose parameters are at par with the best known linear codes.

	\section{Notation and Preliminaries}
	Throughout the article, we use the notation which is in accordance with \cite{CM25}. For better accessibility, we restate the notation and include some fundamental results in this section. 
	
	Let $\mathbb{F}_q$ denote the field with $q$ elements and let $\mathbb{F}_qG$ be the finite semisimple group algebra of a group $G$ over $\mathbb{F}_q$, so that $q$ is relatively prime to $|G|$, the order of $G$. For $\alpha = \sum_{g \in G} \alpha_g g \in \mathbb{F}_qG$, the weight of $\alpha$ is cardinality of the set $\{ g \in G \mid \alpha_g \neq 0 \}$ and is denoted by  $wt(\alpha)$. The Hamming distance between $\alpha$ and $\beta=\sum_{g \in G} \beta_g $ in $\mathbb{F}_qG$ is $d(\alpha,\beta) = |\{ g \in G \mid \alpha_g \neq \beta_g \}|$, which satisfies $d(\alpha,\beta) = wt(\alpha - \beta)$, and hence $wt(\alpha) = d(\alpha,0)$. The weight of an ideal $I \subseteq \mathbb{F}_qG$ is defined as  $\min\{ wt(\alpha) \mid \alpha \in I, \ \alpha \neq 0 \}$. As  stated in the introduction, group codes are nothing but the ideals of $\mathbb{F}_qG$, which are determinable via their idempotents. If $e$ is a pci of $\mathbb{F}_qG$, then $\mathbb{F}_qG e$ is the corresponding central linear $[n,k,d]$ code, where  $n = |G|$, $k = \dim_{\mathbb{F}_q}(\mathbb{F}_qG e)$, the $\mathbb{F}_q$
	dimension of $\mathbb{F}_qG e$ and $d = d(\mathbb{F}_qG e),$ the weight of $\mathbb{F}_qG e$.  
	
	Denote the set of irreducible characters of $G$ over $\mathbb{F}_q$  by $\Irr(G)$.
	If $H$ and $K$ are subgroups of $G$ such that $H/K$ is cyclic, then for a generator $\gamma \in \Irr(H/K)$, the $q$-cyclotomic coset of $\gamma$ is given by $C_q(\gamma)=\{\gamma,\gamma^q,\gamma^{q^2},...,\gamma^{q^{o-1}}\},$ where $o$ is the multiplicative order of $q$ modulo $|H/K|$. Let $\mathcal{C}(H/K)$ be the set of $q$-cyclotomic cosets of $\Irr(H/K)$ containing the generators of $\Irr(H/K)$. The action $g * C = g^{-1} C g, \quad g \in N_G(H) \cap N_G(K),\; C \in \mathcal{C}(H/K),$
	defines the set $\mathcal{R}(H/K)$ of distinct orbits.  Denote the stabilizer of any element of $\mathcal{C}(H/K)$ by $\mathcal{E}_G(H/K).$ For $C = C_q(\chi) \in \mathcal{R}(H/K)$, define 
	\begin{equation}
		\epsilon_C(H,K) = \frac{1}{[H:K]}\,\widehat{K} \sum \limits_{\overline{h} \in H/K} \operatorname{tr}(\chi(\overline{h}))\,h^{-1},
	\end{equation}
	where $\widehat{K} = \frac{1}{|K|}\sum \limits_{k \in K} k$ and $\operatorname{tr} = \operatorname{tr}_{\mathbb{F}_q(\xi_{[H:K]})/\mathbb{F}_q}$ with $\xi_{[H:K]}$ a primitive $[H:K]$-th root of unity. The sum of distinct $G$-conjugates of $\epsilon_C(H,K)$ is denoted $e_C(G,H,K)$.\\
	Consider a pair $(H, K)$ of subgroups of $G$ such that $H$ is normal subgroup of $G$ and $H/K$
	is cyclic as well as a maximal abelian subgroup of  $N_G(K)/K$. Then by (\cite{OdRS04}, Corollary~3.6) and \cite{BdR07}, $(H, K)$ is a strong Shoda pair of $G$ and 
	$e_C(G,H,K)$ is a pci of $\mathbb{F}_qG.$  Further,
	\begin{equation}\label{equation 2}
		\mathbb{F}_qGe_C(G,H,K)\cong M_{[G:H]}(\mathbb{F}_{q^{o/[E:H]}}),
	\end{equation}
	where $E=\mathcal{E}_G(H/K)$ and $o$ is the multiplicative order of $q$ modulo $[H:K].$ \\
	Two strong Shoda pairs of a group are said to be inequivalent, if their corresponding pcis are distinct. We shall denote the set of all  inequivalent strong Shoda pairs of a group $G$ by $\mathcal{S}(G).$   
	Clearly, $(G, G) \in \mathcal{S}(G)$ and for a normal subgroup $K$ of $G$, the pair $(G, K) \in \mathcal{S}(G)$ if and only if $G/K$ is cyclic. The following theorem provides parameters of the codes associated with the pcis  corresponding to strong Shoda pair of type $(G,K)$.
	\begin{theorem}\label{parmameters $G=H$}
		Let $\mathbb{F}_q G$ be a finite semisimple group algebra. If $K$ is a normal subgroup of $G$ such that $G/K=\langle gK\rangle$, for some $g
		\in G$, then the code corresponding to the \linebreak pci(s) $e:=e_C(G, G, K)$,  $C \in \mathcal{R}(G/K)$, satisfy the following:
		\begin{itemize}
			\item[(i)] $\dim_{\mathbb{F	}_q}(\mathbb{F}_q G e) = o_{|G/K|}(q)$;
			\item[(ii)] The set $\mathcal{B}:= \{e, eg, \ldots, eg^{o_{|G/K|}(q)-1}\}$ is an $\mathbb{F}_q$-basis for  $\mathbb{F}_q G e$;
			\item[(iii)] $2|K| \leq d \leq \operatorname{wt}(e)$,  where $d$ denotes the minimum distance of the code and $\operatorname{wt}(e)$ denotes the weight of the idempotent $e$.
			
		\end{itemize}
		In particular, if $|G/K|=p^j$ where $j \in \mathbb{N}$ and  $p$ is an odd prime such that $o_{p^j}(q)=\phi(p^j)$,  then the minimum distance of code generated by $e$ is $2|K|$.
		
	\end{theorem}
	\begin{proof}
		As stated in preliminaries, we have that $\mathbb{F}_qGe_C(G,H,K)\cong M_{[G:H]}(\mathbb{F}_{q^{o/[E:H]}}),$
		where $E=\mathcal{E}_G(H/K)$ and $o$ is the multiplicative order of $q$ modulo $[H:K].$  So $\mathbb{F}_qGe_C(G, G, K)\cong \mathbb{F}_{q^{o_{|G/K|}(q)}}$ and we get the desired dimension.  
		Now we verify that $\mathcal{B}$ is an $\mathbb{F}_q$-basis for  $\mathbb{F}_q G e$, where  $e:=e_C(G, G, K)$.
		By dimension consideration, it is sufficient to show that the set $\mathcal{B}$ is linearly independent. Let $|G/K|=k$.  If $\mathcal{B}$ is a linearly dependent set, then
		\begin{equation}
			\sum\limits_{\mu_g=0}^{o_k(q)-1}\beta_{\mu_g}g^{\mu_g}e=0,
		\end{equation} 
		for some non-zero coefficients $\beta_{\mu_g},$ which yields that $ge$ is a root of non-zero polynomial of degree at most $o_{k}(q)-1$. This is not possible because $\mathbb{F}_qGe$ is the smallest field containing $ge$ and has degree   $o_{k}(q)$ over $\mathbb{F}_q$. 
		For distance bound firstly, observe that $\widehat{K}e_C(G, G, K)=e_C(G, G, K)$, so that $\mathbb{F}_qGe_C(G, G, K)\subseteq\mathbb{F}_qG\widehat{K}$. 
		Any element $\alpha \in \mathbb{F}_q G e_C(G, G, K)$ can be written as $\alpha = \left( \sum \limits_{t \in T} \alpha_t t \right) \widehat{K}$, with $\alpha_t \in \mathbb{F}_q$, where $T$ denotes the transversal of $K$ in $G$. If only one coefficient $\alpha_t$ is non-zero, say $\alpha = \alpha_t t \widehat{K}$ for some $t \in T$, then $\mathbb{F}_q G e_C(G, G, K) \supseteq \mathbb{F}_q G \widehat{K}$, which implies $k=o_k(q)$,  a contradiction. Hence, at least two coefficients must be non-zero, implying that each non-zero codeword has weight at least $2|K|$.\\
		Now if $|G/K|=p^j$ such that $o_{p^j}(q)=\phi(p^j)$ then from (\ref{equation 3}), the expression for $e=\widehat{K}[1-\widehat{\langle g^{p^{j-1}}\rangle }]$ and  $(1-g^{p^{j-1}})e=(1-g^{p^{j-1}})\widehat{K}$, which implies $d \leq 2|K|$.
			\end{proof}
			The above result is proved in a general setting for an arbitrary finite group $G$.  Henceforth, we focus on codes generated by the pcis corresponding to strong Shoda pairs $(H,K)$ where $H$ is a proper subgroup of $G$.\\
		It may be noted that metacyclic groups are normally monomial and hence the algorithms given in \cite{BM14} and \cite{BM16} to compute $\mathcal{S}(G)$ are applicable for these groups.  
		Let $G$ be a metacyclic group of the form $C_{p_1^m} \rtimes C_{p_2^l}$, where $p_1$ and $p_2$ are distinct primes, and $C_{p_2^l}$ acts faithfully on $C_{p_1^m}$. Then,  $G$ can be presented as
	\begin{equation}\label{equation 3}
		G = \langle a,b \mid a^{p_1^m} = b^{p_2^l} = 1,\ b^{-1}ab = a^r \rangle,
	\end{equation} 
	where $m,l,r \in \mathbb{N}$ are such that $o_{p_1^m}(r) = p_2^l$.  
	By  \cite{JOdRV13} we have that,   \begin{equation}\label{equation 4}
		\mathcal{S}(G) = \{(G,G)\} \cup \{(G,\langle a,b^{p_2^{j_2}} \rangle) \mid j_2 = 1,\dots,l\} \cup \{(\langle a\rangle, \langle a^{p_1^{j_1}}\rangle) \mid j_1 = 1,\dots,m\}.
	\end{equation}
	The following lemma, analogous to~\cite[Lemmas~3.1 and~3.2]{CM25} shall be useful in the study of metacyclic group codes of even length, particularly for computing traces.
	\begin{lemma}\label{lemma_trace_zero 2}  
	For $i, q  \in \mathbb{N}$, where $q$ is a power of some odd prime, we have the following:
	\begin{enumerate}
		\item[(1)]  If $q=1 + 2^{i_0} c $, with $c$ odd and $i_0\geq 2$, then  $
		o_{2^i}(q) = 
		\begin{cases}
			2^{i - i_0 }, & ~\mathrm{if } ~i > i_0 \\
			1, & ~\mathrm{otherwise} \\
			
		\end{cases} ~~{\mathrm {and}}
		$
		   
		  $$	\tr(\xi_{2^i}) =\sum \limits_{j=0}^{o_{2^i}(q)-1}\xi_{2^i}^{q^j}= 0, \quad \text{if and only if } i > i_0.$$
		\item[(2)] If $q=-1 + 2^{i_0} c $, with $c$ odd and $i_0\geq 2$, then $
		o_{2^i}(q) = 
		\begin{cases}
			2^{i - i_0 }, & ~~\mathrm{if}~ i > i_0 \\
			2, & ~~\mathrm{otherwise} \\
			
		\end{cases}~~{\mathrm {and}}
		$
		
		$$	\tr(\xi_{2^i}) =\sum \limits_{j=0}^{o_{2^i}(q)-1}\xi_{2^i}^{q^j}= 0, \quad \text{if and only if } i >  i_0~ \mathrm{or}~i=2.$$
	\end{enumerate}
\end{lemma}

\begin{proof} If $q$ is a power of an odd prime and $i\in \mathbb{N}$, then
	the expression of order $o_{2^i}(q)$ as in the statement follows from (\cite{SBR07}, Section 2.2). \\
	Let  \( q = 1+2^{i_0}c  \), where \( c \) is an odd integer and  $i_0\geq 2$. Clearly, if $i \leq i_0$, then $o_{2^i}(q)=1$  and $\text{tr}(\xi_{2^i}) \neq  0.$
	Suppose $i > i_0$, so that  $o_{2^i}(q) = 2^{i - i_0 }$. We see that the sets $J:=\{ q^j \mid 0 \leq j < o_{2^i}(q) \}$
	and  \(K:=\{ 1 + 2^{i_0}k \mid 0 \leq k < 2^{i-i_0} \} \) contain same elements  modulo \( 2^i \). This is because the cardinalities of the sets $J$ and $K$ are same and for any \( q^j \in J \),  i.e., $0 \leq j < 2^{i - i_0 }$,  we can write
	\[
	q^j= (1+2^{i_0}c )^j \equiv 1 + 2^{i_0}k_j \mod 2^i,
	\]
	for some \( 0 \leq k_j < 2^{i-i_0} \), so that $1 + 2^{i_0}k_j \in K$ . Hence, 
	\[
	\tr(\xi_{2^i}) = \sum \limits_{j=0}^{o_{2^i}(q)-1}\xi_{2^i}^{q^j}=\sum \limits_{k=0}^{2^{i-i_0}-1} \xi_{2^i}^{1 + 2^{i_0}k }=\xi_{2^i} \sum\limits_{k=0}^{2^{i-i_0}-1} \left( \xi_{2^i}^{2^{i_0}} \right)^k=0,
	\] as $ \xi_{2^i}^{2^{i_0}}  \neq 1 $ when $i > i_0$. 
	Therefore,
	$
	\tr(\xi_{2^i}) = 0, ~\text{if~and ~only ~if} ~i > i_0.
	$
	\\
	Now, if \( q = -1 + 2^{i_0}c \), with \( c \) odd and $i_0\geq 2$, then   analogously we obtain that  
	$ \tr(\xi_{2^i}) = 0, $ for $i > i_0$.
	For \(  i \leq  i_0 \), we have  $o_{2^i}(q)=2$, and hence $
	\tr(\xi_{2^i}) = \xi_{2^i} + \xi_{2^i}^{-1}=0 $, if and only if $i=2.$
\end{proof}

		\section{Metacyclic 2-group codes}
		 In this section, we study metacyclic $2$-groups. Specifically, we work with metacyclic $2$-groups  which possess a maximal cyclic subgroup. As per (\cite{Huppert1}, I, Satz 14.9(b)), a metacyclic group of order $2^{n+1}$, where $n \geq 3$ which has a maximal cyclic subgroup, is isomorphic to  one of the following: 
		 	\begin{enumerate}
		\item[(i)] $D_{2^{n+1}} := \langle a, b \mid a^{2^{n}} = 1,\ b^2 = 1,\ b^{-1}ab = a^{-1} \rangle$ (dihedral).
		\item[(ii)] $Q_{2^{n+1}}:= \langle a, b \mid a^{2^{n}} = 1,\ b^2 = a^{2^{n-1}},\ b^{-1}ab = a^{-1} \rangle$  (generalized quaternion).

		\item[(iii)] $SD_{2^{n+1}} := \langle a, b \mid a^{2^{n}} = 1,\ b^2 = 1,\ b^{-1}ab = a^{-1 + 2^{n-1}} \rangle$  (semi-dihedral).
		
		\item[(iv)] $G_{2^{n+1}} := \langle a, b \mid a^{2^{n}} = 1,\ b^2 = 1,\ b^{-1}ab = a^{1 + 2^{n-1}} \rangle$ (ordinary metacyclic).
	\end{enumerate}

	It has been proved in~\cite[Theorem~5]{DFPM09}, that  
	$\mathbb{F}_q D_{2^{n+1}} \ \cong\ \mathbb{F}_q Q_{2^{n+1}}.$ An alternate way to prove this is via theory of strong Shoda pairs. We rather apply this theory to find out when   group algebras $\mathbb{F}_q D_{2^{n+1}}$ and $\mathbb{F}_q SD_{2^{n+1}}$ are isomorphic.  It turns out that $\mathbb{F}_q G_{2^{n+1}}$ is never isomorphic to  $\mathbb{F}_q SG_{2^{n+1}}$.
		\begin{theorem}\label{$F_qD_{2^{n+1}}$ Isomorphic $F_qD_{2^{n+1}}^-$}
		Let $n\geq 3$ and let $q$ be a prime power of some odd prime.
 Then,
		\[
		\mathbb{F}_q D_{2^{n+1}}  (\cong \mathbb{F}_q Q_{2^{n+1}}) \cong \mathbb{F}_q SD_{2^{n+1}}
		\quad \text{if and only if} \quad q \not \equiv -1 \mod{2^{n-1}}.
		\]
	\end{theorem}
	
	\begin{proof}

	A set of strong Shoda pairs of $G$ , where $G=D_{2^{n+1}},~ Q_{2^{n+1}},~ SD_{2^{n+1}},$ computed using algorithm given in  \cite{BM14}, is given by 
	\[
	\mathcal{S}(G)=
	\big\{
	(G, G),\ 
	(G,\langle a \rangle),\ 
	(G, \langle a^2, b \rangle),\ 
	(G, \langle a^2, ab\rangle)
	\big\}
	\cup 
	\big\{
	(\langle a \rangle, \langle a^{2^j} \rangle)\ | \ \ 2 \leq j \leq n
	\big\}.
	\] 
	
	We shall observe that the difference in structure of group algebras possibly occurs, only due to the component corresponding to $(\langle a \rangle, \langle 1 \rangle)$. It follows from  the results stated in Section 2 that if $G=D_{2^{n+1}}$ or $Q_{2^{n+1}}$, then 

		\[
	\mathbb{F}_q G \cong
	\begin{cases}
	4\mathbb{F}_q ~ \bigoplus
		 \limits_{j=2}^{n} 
		\frac{\phi(2^j)}{o_{2^j}(q)} M_2\!\left( \mathbb{F}_{q^{\,o_{2^j}(q)}} \right), ~ \text{if } -1 \in \langle q \rangle \mod{2^n}, \\
		4\mathbb{F}_q ~\bigoplus 
		\limits_{j=2}^{j_0} 
		\frac{\phi(2^j)}{o_{2^j}(q)} \, M_2\!\left( \mathbb{F}_{q^{\,o_{2^j}(q)/2}} \right)
		\bigoplus
		 \limits_{j=j_0+1}^{n} 
		\frac{\phi(2^j)}{2o_{2^j}(q)} M_2\!\left( \mathbb{F}_{q^{\,o_{2^j}(q)}} \right),~\text{otherwise}
	\end{cases}
	\]
	\noindent where $j_0$ is such that  
	\begin{equation}\label{DQ}
		-1 \in \langle q \rangle \mod\ 2^{j_0} ~~\mathrm{but}~ -1 \notin \langle q \rangle \mod\ 2^{j_0+1}
	\end{equation}

	\noindent and \[
	\mathbb{F}_q SD_{2^{n+1}} \cong
	\begin{cases}
	4	\mathbb{F}_q~  \bigoplus
		\limits_{j=2}^{n} 
		\frac{\phi(2^j)}{o_{2^j}(q)}  
		M_2\!\left( \mathbb{F}_{q^{o_{2^j}(q)}} \right),~~ \text{if} -1+2^{n-1} \in \langle q \rangle\mod 2^n\\
		4\mathbb{F}_q~ \bigoplus
		\limits_{j=2}^{j_1} 
		\frac{\phi(2^j)}{o_{2^j}(q)} \, 
		M_2\!\left( \mathbb{F}_{q^{\,o_{2^j}(q)/2}} \right)
		\bigoplus 
		\limits_{j=j_1+1}^{n} 
		\frac{\phi(2^j)}{2o_{2^j}(q)} \,
		M_2\!\left( \mathbb{F}_{q^{\,o_{2^j}(q)}} \right),
		~ \text{otherwise}
		
	\end{cases}
	\]
	where $j_1$ is such that
	
	\begin{equation}\label{SD}
		-1+2^{n-1} \in \langle q \rangle\mod 2^{j_1} ~~\text{but}~
		-1+2^{n-1} \notin \langle q \rangle  \mod\ 2^{j_1+1}.
	\end{equation}

	Note that $j_0\leq n-1$ and if $j_0 <n-1$, then $j_1=j_0$. Now, if $j_0=n-1$, then 	$\mathbb{F}_q G   \cong \mathbb{F}_q SD_{2^{n+1}}$ if and only if $j_1=j_0$ and the conditions $	-1 \notin \langle q \rangle  \mod\ 2^{n}$ and $	-1+2^{n-1} \notin \langle q \rangle  \mod\ 2^{n}$ are equivalent. We prove that
	\begin{equation}\label{eq 8}
		-1 \in \langle q \rangle \mod{2^n} 
		\iff 
		-1 + 2^{n-1} \in \langle q \rangle \mod{2^n}
	\end{equation}

precisely when $ q \not \equiv -1 \mod{2^{n-1}}$.

Observe that  $U(2^n) \cong C_2 \times C_{2^{n-2}} \cong \langle -1 \rangle \times \langle 5 \rangle$, and hence $U(2^n)$ contains exactly three elements of order $2$, namely $-1$, $-1+2^{n-1}$ and $1-2^{n-1}$. 
	
	\smallskip
	\noindent\textbf{Case (i):} Suppose $q \equiv 1 \mod{4}$.  
	Then $\langle q \rangle = \langle 5^k \rangle$ for some $k$. Clearly, $-1 \notin \langle q \rangle$.  
	Since $5^{2^{n-3}} \equiv 1 - 2^{n-1} \mod{2^n}$, we have $1 - 2^{n-1} \in \langle 5 \rangle$.  
	Consequently, $-1 \cdot (1 - 2^{n-1}) = -1 + 2^{n-1} \in \langle -5 \rangle \not\subseteq \langle q \rangle$.  
	Hence, (\ref{eq 8}) holds.
	
	\smallskip
	\noindent\textbf{Case (ii):} Suppose $q \equiv -1 \mod{4}$.  
	We claim that (\ref{eq 8}) holds if and only if $q \not \equiv -1 \mod{2^{n-1}}$.  
	
	First, assume $q \equiv -1 \mod{2^{n-1}}$.  
	Then $q^2 \equiv 1 \mod{2^n}$, so $\langle q \rangle$ has only one element of order $2$. If this element is either $-1$ or $-1+2^{n-1}$, (\ref{eq 8}) does not hold, otherwise we must have the order $2$ element to be $q=1-2^{n-1}$ which  contradicts the assumption  $q \equiv -1 \mod{2^{n-1}}$.

	Conversely, if $q \not\equiv -1 \mod{2^{n-1}}$, then by \Cref{lemma_trace_zero 2} we have $o_{2^n}(q) > 2$.  
	Thus $\langle q \rangle$ contains at least two elements of order $2$, which forces the third element of order $2$ to also lie in $\langle q \rangle$.  
	Hence, the claim holds.
\end{proof}

	Since $\mathbb{F}_q D_{2^{n+1}}$ and $\mathbb{F}_q Q_{2^{n+2}}$ are always isomorphic, and $\mathbb{F}_q D_{2^{n+1}} \cong \mathbb{F}_q SD_{2^{n+1}}$ if and only if $ q \not \equiv -1 \mod 2^{n-1}$, it follows from the proof of  \Cref{$F_qD_{2^{n+1}}$ Isomorphic $F_qD_{2^{n+1}}^-$} that, except for the component corresponding to the strong Shoda pair $(\langle a \rangle, 1)$, the codes generated by the pcis of $\mathbb{F}_q D_{2^{n+1}}$ and $\mathbb{F}_q SD_{2^{n+1}}$ are equivalent. Moreover, the computation of the pcis corresponding to $(\langle a \rangle, 1)$ is somewhat similar in both $\mathbb{F}_q D_{2^{n+1}}$ and $\mathbb{F}_q SD_{2^{n+1}}$. Therefore, we consider the codes generated by the pcis of $\mathbb{F}_q D_{2^{n+1}}$  and $\mathbb{F}_q G_{2^{n+1}}$ only.\\
	
			We are now in position to write the pcis of the groups under consideration.
	
	\subsection{\underline{$\mathbf{D_{2^{n+1}}}$}}

	$D_{2^{n+1}} := \langle a, b \mid a^{2^{n}} = 1,\ b^2 = 1,\ b^{-1}ab = a^{-1} \rangle$ with $n \geq 3.$
	\begin{theorem}\label{prop_Dihedral}
		Let	$\mathbb{F}_q$ be a finite field containing $q$ elements, where $q$ is power of some odd prime so that $q$ is of the form $q= \pm  1 + 2^{i_0}c$, where $c$ is odd and $i_0 \geq 2$. The pcis of $\mathbb{F}_qD_{2^{n+1}}$ are as in \Cref{tab3:my_label}.
		\begin{small}
			\begin{table}[hbt!]

				\begin{tabular}{|ccl|}			
					\hline 
					
					\hline
					&&	if ~$q= \pm 1+2^{i_0}c$ with $c$ odd. \\
					
					$e_1$ &$:=$&$e_C(D_{2^{n+1}},D_{2^{n+1}},D_{2^{n+1}}), C\in \mathcal{R}(D_{2^{n+1}}/D_{2^{n+1}}) $\\
					
					&$=$ &$\widehat{D_{2^{n+1}}}$   \\&&\\

					$e_{2}$ &$:= $ & $e_C(D_{2^{n+1}},D_{2^{n+1}}, \langle a\rangle),C\in \mathcal{R}(D_{2^{n+1}}/ \langle a\rangle)$ \\

					&$=$&$\widehat{\langle a \rangle}-\widehat{D_{2^{n+1}}}$
					\\&&\\
										
					$e_{3}$ &$:= $ & $e_C(D_{2^{n+1}},D_{2^{n+1}}, \langle a^2,b\rangle),C\in \mathcal{R}(D_{2^{n+1}}/ \langle a^2,b\rangle)$\\
					
					&$=$&$\widehat{\langle a^2,b \rangle }-\widehat{D_{2^{n+1}}}$ \\&&\\

					$e_{4}$ &$:= $ & $e_C(D_{2^{n+1}},D_{2^{n+1}}, \langle a^2,ab\rangle),C\in \mathcal{R}(D_{2^{n+1}}/ \langle a^2,ab\rangle)$\\
					&$=$&$\widehat{\langle a^2,ab \rangle}-\widehat{D_{2^{n+1}}}$ \\

					\hline 	
					
					\hline 
					&&	if ~$q= 1+2^{i_0}c$ with $c$ odd. \\

					$e_{2^{j},k}$ &$:= $ & $e_C(D_{2^{n+1}},\langle a \rangle,\langle a^{2^{j}} \rangle),C=C_q(\gamma^k)\in \mathcal{R}(\langle a \rangle/\langle a^{2^{j}} \rangle),2 \leq j \leq n$ \\
					&$=$&$\begin{cases}
						\frac{1}{2^{j}} \widehat{\langle a^{2^{j }}\rangle}\sum\limits_{\mathfrak{i}=0}^{2^{j}-1} [\tr(\xi_{2^{j}}^{k\mathfrak{i}})]a^{-\mathfrak{i}},~~~~~~~~~~~~~~~\mathrm{if} ~ 1\leq j \leq i_0\\
						\frac{1}{2^{j}} \widehat{\langle a^{2^{j }}\rangle}\sum\limits_{\mathfrak{i}'=0}^{2^{i_0}-1}[\tr(\xi_{2^{j}}^{k\mathfrak{i}'{2^{j-i_0}}})]a^{-\mathfrak{i}'{2^{j-i_0}}},~~\mathrm{otherwise}.
					\end{cases}$
					\\
%
%
%

					\hline
					
					\hline 
					
					&&	if ~$q= -1+2^{i_0}c$ with $c$ odd. \\
					\underline{\textbf{$-1\in\langle q\rangle\mod 2^j $}}&&\\

					$e_{2^{j},k}$ &$:= $ & $e_C(D_{2^{n+1}},\langle a \rangle,\langle a^{2^{j}} \rangle),C=C_q(\gamma^k)\in \mathcal{R}(\langle a \rangle/\langle a^{2^{j}} \rangle),2\leq j \leq n$ \\
					
					&$=$&$\begin{cases}
						\frac{1}{2^{j}} \widehat{\langle a^{2^{j }}\rangle}\sum\limits_{\mathfrak{i}=0}^{2^{j}-1} [\tr(\xi_{2^{j}}^{k\mathfrak{i}})]a^{-\mathfrak{i}},~~~~~~~~~~~~~~~\mathrm{if} ~ 3\leq j \leq i_0\\
						\widehat{\langle a^2 \rangle }-	\widehat{\langle a^4 \rangle },~~~~~~~~~~~~~~~~~~~~~~~~~~~~~~\mathrm{if}~ j=2\\
						\frac{1}{2^{j}} \widehat{\langle a^{2^{j }}\rangle}\sum\limits_{\mathfrak{i}'=0}^{2^{i_0}-1}[\tr(\xi_{2^{j}}^{k\mathfrak{i}'{2^{j-i_0}}})]a^{-\mathfrak{i}'{2^{j-i_0}}},~ \mathfrak{i}'\neq 2^{j-2},~ 3\cdot 2^{j-2}	,~\mathrm{if} ~~j>i_0.
					\end{cases}$
					\\

					\underline{\textbf{$-1\not\in\langle q \rangle \mod 2^j$}}&&\\
					
					$e_{2^{j},k}$ &$:= $ & $e_C(D_{2^{n+1}},\langle a \rangle,\langle a^{2^{j}} \rangle),C=C_q(\gamma^k)\in \mathcal{R}(\langle a \rangle/\langle a^{2^{j}} \rangle),2\leq j \leq n$ \\
					
					&$=$&$\begin{cases}
						\frac{1}{2^{j}} \widehat{\langle a^{2^{j }}\rangle}\sum\limits_{\mathfrak{i}=0}^{2^{j}-1} [\tr(\xi_{2^{j}}^{k\mathfrak{i}}) +\tr(\xi_{2^{j}}^{-k\mathfrak{i}})]a^{-\mathfrak{i}},~~~~~~~~~~~~~~~\mathrm{if} ~ 3 \leq j \leq i_0, ~ \mathfrak{i}\neq 2^{j-2},~ 3\cdot 2^{j-2}\\
						\widehat{\langle a^2 \rangle }-	\widehat{\langle a^4 \rangle },~~~~~~~~~~~~~~~~~~~~~~~~~~~~~~~~~~~~~~~~~~~\mathrm{if}~ j=2\\
						\frac{1}{2^{j}} \widehat{\langle a^{2^{j }}\rangle}\sum\limits_{\mathfrak{i}'=0}^{2^{i_0}-1}[\tr(\xi_{2^{j}}^{k\mathfrak{i}'{2^{j-i_0}}}) + \tr(\xi_{2^{j}}^{-k\mathfrak{i}'{2^{j-i_0}}})]a^{-\mathfrak{i}'{2^{j-i_0}}},~\mathfrak{i}'\neq 2^{j-2},~ 3\cdot 2^{j-2}	,~\mathrm{if} ~j>i_0.
					\end{cases}$
					\\
					
					\hline

				\end{tabular}
				
				\caption{Pcis of  $D_{2^{n+1}}$}
				\label{tab3:my_label}
				
					For $2 \leq j \leq n$, 
				the possible choices of $k$ yield $\tfrac{\varphi(2^j)}{o_{2^j}(q)}$ distinct idempotents  when $-1 \in \langle q\rangle \mod{2^j}$ and $\tfrac{\varphi(2^j)}{2o_{2^j}(q)}$ distinct idempotents when $-1 \notin \langle q\rangle \mod{2^j}$, $\varphi$ being the Euler totient function.
				
			\end{table}
		
		\end{small}

			\end{theorem}	
			\begin{proof}
					If $q \equiv 1\mod 4$, then (\cite{CM25}, Proposition 4.1) holds for any prime $p$ (including $p=2$). This is because, in this case, 	in view of \Cref{lemma_trace_zero 2}, the result in (\cite{CM25}, Lemma 3.2) holds for any prime $p$ (not necessarily odd).\\
					Hence, assuming  $q \equiv -1 \mod 4$, we write the pcis of $\mathbb{F}_qD_{2^{n+1}}$ corresponding to $(\langle a \rangle,\langle a^{2^{j}} \rangle) \in \mathcal{S}(D_{2^{n+1}})$, where  $1\leq j\leq n$, i.e., $e_C(D_{2^{n+1}},\langle a \rangle, \langle a^{2^{j}} \rangle)$, where $C\in \mathcal{R}(\langle a\rangle/\langle a^{2^{j}} \rangle)$ for   $1\leq j \leq n$.
	If $-1\in \langle q\rangle$ then 
			$\mathcal{R}(\langle a \rangle/ \langle a^{2^{j}} \rangle)=\mathcal{C}(\langle a \rangle/ \langle a^{2^{j}} \rangle)$ 
			and
			 in this case  we have  
				$$\begin{aligned}
					e_C(G,\langle a \rangle,\langle a^{2^{j}} \rangle)=\epsilon_C(\langle a \rangle,\langle a^{2^{j}} \rangle)
					=\frac{1}{2^{j}}\widehat{\langle a^{2^{j}} \rangle}\Sigma_{\mathfrak{i}=0}^{2^{j}-1}\tr(\xi_{2^{j}}^\mathfrak{i}})){a^{-\mathfrak{i}},
				\end{aligned}$$
					where
					$$\tr(\xi_{2^{j}}^{k\mathfrak{i}})=(\xi_{2^{j}}^{k\mathfrak{i}})^{q^0}+(\xi_{2^{j}}^{k\mathfrak{i}})^q+(\xi_{2^{j}}^{k\mathfrak{i}})^{q^2}+....+(\xi_{2^{j}}^{k\mathfrak{i}})^{q^{o_{2^j}(q)-1}}.$$
				
				Now, if $1\leq j \leq i_0$, $j \neq 2$ then $\tr(\xi_{2^{j}}^{k\mathfrak{i}})\neq 0$,  by \Cref{lemma_trace_zero 2} and if $j=2, $ then 
				$$\begin{aligned}
					\epsilon_C(\langle a \rangle , \langle a^4 \rangle)&=\frac{1}{4}\widehat{\langle a^4 \rangle}\sum \limits_{\mathfrak{i}=0}^{3}\tr(\gamma^{k}(\overline{a^\mathfrak{i}})){a^{-\mathfrak{i}}}
					=	
					\widehat{\langle a^4 \rangle}-\widehat{\langle a^2 \rangle}.
				\end{aligned}$$
					Assume $i_0<j\leq m$. In this case, $\xi_{2^{j}}^{k\mathfrak{i}}$ is $2^{i}$-th primitive root of unity, if $\gcd(\mathfrak{i},2^{j})=2^{j-i}$. 
					 Hence,  $$\tr(\xi_{2^{j}}^{k\mathrm{i}})=\frac{o_{2^{j}}(q)}{o_{2^{i}}(q)}[(\xi_{2^{i}})^{q^0}+(\xi_{2^{i}})^q+(\xi_{2^{i}})^{q^2}+....+(\xi_{2^{i}})^{q^{o_{2^{i}}(q)-1}}].$$
						Note that by \Cref{lemma_trace_zero 2}, the above term is zero if and only if  $i>i_0$ or $i=2$. Therefore, the terms which do not vanish are precisely the ones where $i_0\geq  i$ and $i \neq 2$, i.e., $\mathfrak{i}$ is a multiple of $2^{j-i_0}$ but not an odd multiple of $2^{j-2}$. So we obtain $$e_C(D_{2^{n+1}},\langle a \rangle,\langle a^{2^{j}} \rangle)=  \frac{1}{2^{j}} \widehat{\langle a^{2^{j }}\rangle}\sum\limits_{\mathfrak{i}'=0}^{2^{i_0}-1}[\tr(\xi_{2^{j}}^{tk\mathfrak{i}'{2^{j-i_0}}})]a^{-\mathfrak{i}'{2^{j-i_0}}}, \mathfrak{i}'\neq 2^{j-2},~ 3\cdot 2^{j-2}.$$
					If $-1\notin\langle q\rangle\mod{2^j}$, then
				\[
				\begin{aligned}
					e_C(D_{2^{n+1}},\langle a\rangle,\langle a^{2^{j}}\rangle)
					&=\frac{1}{2^{j}}\widehat{\langle a^{2^{j}}\rangle}
					\sum\limits_{\mathfrak{i}=0}^{2^{j}-1}\big[\tr(\xi_{2^{j}}^{k\mathfrak{i}})+\tr(\xi_{2^{j}}^{-k\mathfrak{i}})\big]a^{-\mathfrak{i}}\\
					&=\frac{1}{2^{j}}\widehat{\langle a^{2^{j}}\rangle}
					\sum\limits_{\mathfrak{i}=0}^{2^{j}-1}\tr\big(\xi_{2^{j}}^{k\mathfrak{i}}+\xi_{2^{j}}^{-k\mathfrak{i}}\big)\,a^{-\mathfrak{i}}.
				\end{aligned}
				\]
				
				Note that \(\xi_{2^{j}}^{k\mathfrak{i}}+\xi_{2^{j}}^{-k\mathfrak{i}}\neq 0\) precisely when
				\(\mathfrak{i}\neq 2^{j-2},\,3\cdot 2^{j-2}\).
			Therefore,	by \Cref{lemma_trace_zero 2}, we have 
				
				\begin{itemize}
					\item[(i)] If \(1\le j\le i_0\), then
					\[
					e_C(D_{2^{n+1}},\langle a\rangle,\langle a^{2^{j}}\rangle)
					=\frac{1}{2^{j}}\widehat{\langle a^{2^{j}}\rangle}
					\sum\limits_{\substack{\mathfrak{i}=0\\ \mathfrak{i}\neq 2^{j-2},\,3\cdot2^{j-2}}}^{2^{j}-1}
					\tr\big(\xi_{2^{j}}^{k\mathfrak{i}}+\xi_{2^{j}}^{-k\mathfrak{i}}\big)\,a^{-\mathfrak{i}},
					\]
					and for these \(\mathfrak{i}\) the trace is non-zero.
					
					\item[(ii)] If \(i_0<j\le m\), write \(\mathfrak{i}=\mathfrak{i}'2^{\,j-i_0}\) with
					\(\mathfrak{i}'=0,\dots,2^{i_0}-1\). Then the sum reduces to
					\[
					e_C(D_{2^{n+1}},\langle a\rangle,\langle a^{2^{j}}\rangle)
					=\frac{1}{2^{j}}\widehat{\langle a^{2^{j}}\rangle}
					\sum\limits_{\mathfrak{i}'=0}^{2^{i_0}-1}
					\tr\big(\xi_{2^{j}}^{k\mathfrak{i}'2^{\,j-i_0}}+\xi_{2^{j}}^{-k\mathfrak{i}'2^{\,j-i_0}}\big)\,a^{-\mathfrak{i}'2^{\,j-i_0}},
					\]
					where the vanishing indices correspond to \(\mathfrak{i}'=2^{\,i_0-2}\) and
					\(\mathfrak{i}'=3\cdot 2^{\,i_0-2}\). Thus the effective summation excludes
					\(\mathfrak{i}'=2^{\,i_0-2},\,3\cdot 2^{\,i_0-2}\), and for the remaining \(\mathfrak{i}'\)
					the trace is non-zero.
				\end{itemize}
				 \end{proof}
					Using  \Cref{prop_Dihedral} we obtain $\dim_{\mathbb{F}_q}(\mathbb{F}_q D_{2^{n+1}} e_{2^j,k})$, an  $\mathbb{F}_q$-basis of $\mathbb{F}_q D_{2^{n+1}} e_{2^j,k}$ and bounds on the distance $d:=d(\mathbb{F}_qD_{2^{n+1}}e_{2 ^{j },k})$ exactly as obtained in Corollaries 4.3 and 4.5 of \cite{CM25} with $p=2$. The parameters and the basis remain unchanged except that the upper bound improves in the current case because of the reduced support.

	\begin{remark}
		It is worth noting that the dihedral codes discussed in \cite{GR22b}, \cite{GR22a} and \cite[Section~4]{DFPM09} can be obtained as special cases of results in this subsection, when the group under consideration is a dihedral group of order $2^{n+1}$. Each of these works considers different conditions on the field size $q$: $q$ is of the form $8c \pm 1$ with $c$ odd in \cite{GR22a}, $o_{2^n}(q) = 1$ or $2$ in \cite{GR22b}, and $o_{2^n}(q) = \phi(2^n)$ in \cite{DFPM09}. Consequently, the corresponding results on code dimension and minimum distance in those works follow as special cases. In \Cref{prop_Dihedral}, we provide a unified treatment of these computations for all such choices of $q$.
	\end{remark}
	 It has also been observed that non-central group codes play an equally significant role; in fact, they often yield codes that are inequivalent to abelian codes and may even possess better parameters. Such codes correspond to left (or right) ideals of the group algebra, and every left (right) ideal $I \subseteq \mathbb{F}_q G$ can be generated by a suitable idempotent. A complete set of pairwise orthogonal irreducible left idempotents of $\mathbb{F}_qD_{2^{n+1}}$ follows from  above theorem and  (\cite{APM19}, Proposition~2.5).

\begin{corollary}\label{left pcis F_qD_{2^{n+1}}}
	The semisimple group algebra $\mathbb{F}_qD_{2^{n+1}}$ decomposes into minimal left ideals generated by a complete set of primitive orthogonal idempotents, given by:

	\begin{itemize}
		\item[(i)] Four central idempotents: $e_1, e_2, e_3, e_4$.
		\item[(ii)] $2 \dfrac{\phi(2^j)}{\kappa o_{2^j}(q)}$ left  idempotents: $e_{2^j,k}\widehat{\langle b \rangle}$ and $e_{2^j,k}(1-\widehat{\langle b \rangle})$ with $2 \leq j \leq n$,	where $\kappa=2$ for $r+1\leq j \leq n$, if $r$ is such that $-1 \in \langle q \rangle \mod{2^r}$ but $-1 \notin \langle q \rangle \mod{2^{r+1}}$ and $\kappa=1$ in all other cases.
	\end{itemize}
\end{corollary}

We next consider the group algebra $\mathbb{F	}_qG_{2^{n+1}}$.
	\subsection{\underline{$\mathbf{G_{2^{n+1}}}$}}
$G_{2^{n+1}} := \langle a, b \mid a^{2^{n}} = 1,\ b^2 = 1,\ b^{-1}ab = a^{1+2^{n-1}} \rangle$, $n \geq 3.$ 

	\begin{theorem}\label{prop_G_{2^{n+1}}}
		Let	$\mathbb{F}_q$ be a finite field containing $q$ elements, where $q$ is power of some odd prime so that $q$ is of the form $q= \pm  1 + 2^{i_0}c$, where $c$ is odd and $i_0 \geq 2$. The pcis of $\mathbb{F}_qG_{2^{n+1}}$ are as in \Cref{tab4:my_label}.
		
		\begin{small}

			\begin{table}

				\begin{tabular}{|ccl|}
					
					\hline 
					
					\hline
					&&	if ~$q= \pm 1 +2^{i_0}c$ with $c$ odd. \\
					
					$e_1$ &$:=$&$e_C(G_{2^{n+1}}, G_{2^{n+1}}, G_{2^{n+1}}), C\in \mathcal{R}(G_{2^{n+1}}/G_{2^{n+1}}) $\\
					
					&$=$ &$\widehat{G_{2^{n+1}}}$   \\&&\\

					$e_{2}$ &$:= $ & $e_C(G_{2^{n+1}},G_{2^{n+1}}, \langle a\rangle),C\in \mathcal{R}(G_{2^{n+1}}/ \langle a\rangle)$ \\

					&$=$&$\widehat{\langle a \rangle}-\widehat{G_{2^{n+1}}}$
					\\&&\\

					$e_{3}$ &$:= $ & $e_C(G_{2^{n+1}},G_{2^{n+1}}, \langle a^2,b\rangle),C\in \mathcal{R}(G_{2^{n+1}}/ \langle a^2,b\rangle)$\\
					
					&$=$&$\widehat{\langle a^2,b \rangle }-\widehat{G_{2^{n+1}}}$ \\&&\\

					$e_{4}$ &$:= $ & $e_C(G_{2^{n+1}}, G_{2^{n+1}}, \langle a^2,ab\rangle),C\in \mathcal{R}(G_{2^{n+1}}/ \langle a^2,ab\rangle)$\\
					&$=$&$\widehat{\langle a^2,ab \rangle}-\widehat{G_{2^{n+1}}}$ \\

					\hline 
					
					\hline 
					&&	if ~$q= 1+2^{i_0}c$ with $c$ odd. \\

					$e_{2^{j},k}$ &$:= $ & $e_C(G_{2^{n+1}}, G_{2^{n+1}}, K),C=C_q(\gamma^k)\in \mathcal{R}(G_{2^{n+1}}/K)$,\\ &&where $K=\langle a^{2^{j}}, b \rangle$ or $\langle a^{2^{{j}-1}}b\rangle,~2\leq j \leq n-1$ \\
					
					&$=$&$\begin{cases}
						\frac{1}{2^{j}} \widehat{K}\sum\limits_{\mathfrak{i}=0}^{2^{j}-1} [\tr(\xi_{2^{j}}^{k\mathfrak{i}})]a^{-\mathfrak{i}},~~~~~~~~~~~~~~~\mathrm{if} ~ 2\leq j \leq i_0\\
						\frac{1}{2^{j}} \widehat{K}\sum\limits_{\mathfrak{i}'=0}^{2^{i_0}-1}[\tr(\xi_{2^{j}}^{k\mathfrak{i}'{2^{j-i_0}}})]a^{-\mathfrak{i}'{2^{j-i_0}}},~\mathrm{otherwise}.
					\end{cases}$
					\\

					&&\\
					
					$e_{2^n,k}$ &$:= $ & $e_C(G_{2^{n+1}},\langle a \rangle,\langle 1 \rangle),C=C_q(\gamma^{k})\in \mathcal{R}(\langle a \rangle/\langle 1 \rangle)$ \\
					&$=$&$\begin{cases}
						\frac{1}{2^{i_0}} \sum\limits_{\mathfrak{i}=0}^{2^{i_0}-1} [\tr(\xi_{2^{i_0}}^{k\mathfrak{i}}) ]a^{-\mathfrak{i}},~~~~~~~~~~~~~\mathrm{if}~n=i_0\\
						\frac{1}{2^{n}}\Sigma_{\mathfrak{i}'=0}^{2^{i_0}-1}[\tr(\xi_{2^{n}}^{{k\mathfrak{i}'2^{n-i_0}}})]a^{-\mathfrak{i}'{2^{n-i_0}}}~\mathrm{if}~n>i_0.
					\end{cases}$
					\\
%

					\hline
					
					\hline 
					&&	if ~$q= -1+2^{i_0}c$ with $c$ odd. \\

					$e_{2^{j},k}$ &$:= $ & $e_C(G_{2^{n+1}}, G_{2^{n+1}}, K),C=C_q(\gamma^k)\in \mathcal{R}(G_{2^{n+1}}/K)$,\\ &&where $K=\langle a^{2^{j}}, b \rangle$ or $\langle a^{2^{{j}-1}}b\rangle,~2\leq j \leq n-1$ \\

					&$=$&$\begin{cases}
						\frac{1}{2^{j}} \widehat{K}\sum\limits_{\mathfrak{i}=0}^{2^{j}-1} [\tr(\xi_{2^{j}}^{k\mathfrak{i}})]a^{-\mathfrak{i}},~~~~~~~~~~~~~\mathrm{if} ~ 2\leq j \leq i_0, j \neq 2\\
						\widehat{\langle a^2, b \rangle }-	\widehat{\langle a^{4}, b \rangle},~~~~~~~~~~~~~~~~~~~\mathrm{if} ~K=\langle a^{4}, b\rangle \\
						
						\widehat{\langle a^{4}b \rangle }-	\widehat{\langle a^2b \rangle },~~~~~~~~~~~~~~~~~~~~~~\mathrm{if}~ K=\langle a^2b\rangle \\
						\frac{1}{2^{j}} \widehat{K}\sum\limits_{\mathfrak{i}'=0}^{2^{i_0}-1}[\tr(\xi_{2^{j}}^{k\mathfrak{i}'{2^{j-i_0}}})]a^{-\mathfrak{i}'{2^{j-i_0}}},~ \mathfrak{i}'\neq 2^{j-2},~ 3\cdot 2^{j-2}	,~\mathrm{if} ~~j>i_0.
					\end{cases}$
					\\
					&&\\
					$e_{2^n,k}$ &$:= $ & $e_C(G_{2^{n+1}},\langle a \rangle,\langle 1 \rangle),C=C_q(\gamma^{k})\in \mathcal{R}(\langle a \rangle/\langle 1 \rangle)$ \\
					\underline{\textbf{$2^{n-1}+1 \in\langle q\rangle\mod 2^n $}}&&\\
					&$=$&$\begin{cases}
						\frac{1}{2^{i_0}} \sum\limits_{\mathfrak{i}=0}^{2^{i_0}-1} [\tr(\xi_{2^{i_0}}^{k\mathfrak{i}}) ]a^{-\mathfrak{i}}, ~~~~~~~~~~~~~~~~~~~~~~~~~~~~~~~~~~~~~~~~~~\mathrm{if}~n=i_0\\
						\frac{1}{2^{n}}\sum\limits_{\mathfrak{i}'=0}^{2^{i_0}-1}[\tr(\xi_{2^{n}}^{{k\mathfrak{i}'2^{n-i_0}}})]a^{-\mathfrak{i}'{2^{n-i_0}}},~ \mathfrak{i}'\neq 2^{j-2},~ 3\cdot 2^{j-2}~~~~~~~~\mathrm{if}~n>i_0.
					\end{cases}$
					\\
					\underline{\textbf{$2^{n-1}+1 \notin\langle q\rangle \mod 2^n$}}&&\\
					&$=$&$\begin{cases}
						\frac{1}{2^{i_0}} \sum\limits_{\mathfrak{i}=0}^{2^{i_0}-1} [\tr(\xi_{2^{i_0}}^{k\mathfrak{i}}) + \tr(\xi_{2^{i_0}}^{(2^{i_0-1}+1)k\mathfrak{i}})]a^{-\mathfrak{i}}, ~~~~\mathfrak{i}~						 \mathrm{is ~even~~~~~~~if}~n=i_0 
\\
						
						\frac{1}{2^{n}}\sum\limits_{\mathfrak{i}'=0}^{2^{i_0}-1}[\tr(\xi_{2^{n}}^{{k\mathfrak{i}'2^{n-i_0}}})+\tr(\xi_{2^{n}}^{{(2^{n-1}+1)k\mathfrak{i}'2^{n-i_0}}})]a^{-\mathfrak{i}'{2^{n-i_0}}},~~~~\mathrm{if}~n>i_0\\~~~~~~~~~~~~~~~~~~~~~~~~~~~~~~~~~~~~~~~~~~~~~~~~~~~~~~~~~~~~~~~~~~~ \mathfrak{i}'\neq 2^{j-2},~ 3\cdot 2^{j-2}
					\end{cases}$
					\\

					\hline

					\hline
				\end{tabular}
				
				\caption{Pcis of  $G_{2^{n+1}}$}
				\label{tab4:my_label}
		
				The possible choices of $k$ yield $\tfrac{\varphi(2^j)}{o_{2^j}(q)}$ distinct idempotents for $2 \leq j \leq n$ when $1+2^{n-1} \in \langle q\rangle \mod{2^j}$, and $\tfrac{\varphi(2^j)}{2o_{2^j}(q)}$ choices when $1+2^{n-1} \notin \langle q\rangle \mod{2^j}$.
	\end{table}		\end{small}

	\end{theorem}	
	\begin{proof}
			The complete list of strong Shoda pairs of $G:=G_{2^{n+1}}$ is given by

				\[
			\mathcal{S}(G) =
			\big\{
			(G, K)  : K\in \{G, \langle a \rangle,\langle a^2, b \rangle,\langle a^2, ab \rangle,  \langle a^{2^j}, b \rangle , \langle a^{2^{j-1}}b \rangle\ | \ \ 2 \leq j \leq n-1\}	\}\cup	\{(\langle a \rangle, \langle 1 \rangle) 
			\big\}.
			\]

All the strong Shoda pairs of  \(G_{2^{n+1}}\), except $(\langle a \rangle, \langle 1 \rangle) $, are of the type 
\((G_{2^{n+1}}, K)\) with \(K\) a proper subgroup of \(G_{2^{n+1}}\). Therefore, in view of \Cref{parmameters $G=H$}
				we consider the idempotent corresponding to  
		$(\langle a \rangle, \langle 1 \rangle)$ only.

		If $2^{\,n-1}+1 \in \langle q \rangle \mod{2^n}$,  then  $e_C(G_{2^{n+1}},\langle a \rangle,\langle 1 \rangle)=\epsilon_C(\langle a \rangle,\langle 1 \rangle)$ 
		and the expression is obtained directly. On the other hand, if $2^{\,n-1}+1 \notin \langle q \rangle \mod{2^n}$, 
		then by the same observation on order $2$ elements in $U(2^n)$ as done  in the proof of 
		\Cref{$F_qD_{2^{n+1}}$ Isomorphic $F_qD_{2^{n+1}}^-$}, 
		this situation occurs precisely when 
		\[
		q \equiv -1 \mod{2^{\,n-1}}
		\quad \text{and} \quad
		q \equiv -1 \mod{4}.
		\]
%
%
%
%
%

By  \Cref{lemma_trace_zero 2}, we obtain that
 for $n=i_0$,  $e_{{2^n}, k}=	\frac{1}{2^{n}} \sum\limits_{\mathfrak{i}=0}^{2^{n}-1} [\tr(\xi_{2^{n}}^{k\mathfrak{i}}) + \tr(\xi_{2^{n}}^{(1+2^{n-1})k\mathfrak{i}})]a^{-\mathfrak{i}} $, where
 $\tr(\xi_{2^{n}}^{k\mathfrak{i}}) + \tr(\xi_{2^{n}}^{(1+2^{n-1})k\mathfrak{i}}) =\tr(\xi_{2^{n}}^{k\mathfrak{i}} + \xi_{2^{n}}^{(2^{n-1}+1)k\mathfrak{i}}) \neq 0$ if and only if $\mathfrak{i}$ is even. 
  And, for $n > i_0$, 
 $e_{{2^n}, k}=	\frac{1}{2^{n}}\sum\limits_{\mathfrak{i}'=0}^{2^{i_0}-1}[\tr(\xi_{2^{n}}^{{k\mathfrak{i}'2^{n-i_0}}})+\tr(\xi_{2^{n}}^{{(1+2^{n-1})k\mathfrak{i}'2^{n-i_0}}})]a^{-\mathfrak{i}'{2^{n-i_0}}},~~ \mathfrak{i}'\neq 2^{j-2},~ 3\cdot 2^{j-2}$ with $\tr(\xi_{2^{n}}^{{k\mathfrak{i}'2^{n-i_0}}}+\xi_{2^{n}}^{{(1+2^{n-1})k\mathfrak{i}'2^{n-i_0}}})\neq 0.$

	\end{proof}
\begin{remark}From $\mathcal{S}(G_{2^{n+1}})$ given in above proof it follows that
\begin{equation}\label{weddG_{2^{n+1}}}
	\mathbb{F}_q G_{2^{n+1}} \cong
	\begin{cases}	4\mathbb{F}_q ~~ \bigoplus	\limits_{j=2}^{n-1} 
		
		2\frac{\phi(2^j)}{o_{2^j}(q)} 
		\mathbb{F}_{q^{\,o_{2^j}(q)}} ~  \bigoplus~
		\frac{\phi(2^n)}{o_{2^n}(q)} 
		M_2\!\left(\mathbb{F}_{q^{\,o_{2^n}(q)/2}}\right),
		~ \text{if $1+2^{n-1} \in \langle q \rangle ~\rm {mod} ~2^n$},
		\\
		4\mathbb{F}_q ~~ \bigoplus	\limits_{j=2}^{n-1} 
		
		2\frac{\phi(2^j)}{o_{2^j}(q)} 
		\mathbb{F}_{q^{\,o_{2^j}(q)}} ~ \bigoplus~
		\frac{\phi(2^n)}{2o_{2^n}(q)} \,
		M_2\!\left(\mathbb{F}_{q^{\,o_{2^n}(q)}}\right),
		~  \text{otherwise}.
	\end{cases}
\end{equation}
\end{remark}

	\begin{corollary}\label{Cor_distance_{M_{2^{n+1}}}}In the foregoing notation, 
			\begin{enumerate}
			\item $\dim_{\mathbb{F}_q}(\mathbb{F}_qG_{2^{n+1}}e_{2^{n},k})=\begin{cases}
				2o_{2^{n}}(q),~\mathrm{if} ~1+2^{n-1} \in \langle q\rangle, \\
				4o_{2^{n}}(q),~\mathrm{if} ~ 1+2^{n-1}\not  \in \langle q\rangle.
			\end{cases}$
			\item 	An $\mathbb{F}_q$-basis of $\mathbb{F}_qG_{2^{n+1}}e_{2^n,k}$ is given by $\mathcal{B}$, where \\
			
			$\mathcal{B}=\begin{cases}
				\{a^{\eta_a}b^{\eta_b}e_{2^{n},k}~|~0\leq \eta_a< o_{2^n}(q),~0\leq \eta_b \leq 1\},~~~~ \mathrm{if}~ 1+2^{n-1}\in \langle q\rangle ,\\
				\{a^{\eta_a}b^{\eta_b}\epsilon_C(\langle a \rangle,\langle a^{2^{n}} \rangle)^t~|~t\in \{1,b\},~0\leq \eta_a< o_{2^n}(q),~0\leq \eta_b \leq 1\}, ~~~\mathrm{if}~ 1+2^{n-1}\not \in \langle q\rangle.
				
			\end{cases}$
			
			\item $d:=d(\mathbb{F}_qG_{2^{n+1}}e_{2 ^{n },k})$ satisfies the following:
			\begin{itemize}
				\item[(i)] if $q=1+2^{i_0}c$, where $c$ is odd and $i_0 \geq 2$, then $2 \leq d \leq 2 ^{ i_0}$,
				\item[(ii)]if $q=-1+2^{i_0}c$, where $c$ is odd and $i_0 \geq 2$, then  $d$ satisfies the following:\\
				
				\underline{\textbf{$2^{n-1}+1 \in\langle q\rangle \mod 2^n$}}~~~~~~	$\begin{cases}	
						\begin{aligned}
						&2 \leq d \leq 2^{i_0}, && \mathrm{if}~ n=i_0>2,\\
						&d=2, &&  \mathrm{if}~ n=2,\\
						&2 \leq d \leq (2^{i_0}-2), &&  \mathrm{if}~ i_0<n
					\end{aligned}

				\end{cases}$
			
			\vspace{0.5cm}

					\underline{\textbf{$2^{n-1}+1 \notin\langle q\rangle \mod 2^n$}}~~~~~~	$\begin{cases}	
					\begin{aligned}
						&2 \leq d \leq 2^{i_0-1}, &&  \mathrm{if}~ n=i_0>2,\\
						
						&2 \leq d \leq (2^{i_0}-2), &&  \mathrm{if}~ i_0<n.
					\end{aligned}

				\end{cases}$
				
%
					\end{itemize}
			
		\end{enumerate}   	
	\end{corollary}

	\begin{corollary}\label{left pcis F_qG_{2^{n+1}}}
		The semisimple group algebra $\mathbb{F}_qG_{2^{n+1}}$ decomposes into minimal left ideals generated by a complete set of primitive orthogonal idempotents, given by:
		\begin{itemize}
			\item[(i)]$4 \;+\; 2 \sum\limits_{j=2}^{n-1}\frac{\phi(2^j)}{o_{2^j}(q)}$ central  idempotents corresponding to strong Shoda pairs of type $(G_{2^{n+1}},K),$ as listed in \Cref{tab4:my_label}.
			\item[(ii)] $2 \dfrac{\phi(2^n)}{\kappa o_{2^n}(q)}$ left  idempotents: $e_{2^n,k}\widehat{\langle b \rangle}$ and $e_{2^n,k}(1-\widehat{\langle b \rangle})$,	where $\kappa=2$ for $1+2^{n-1} \notin \langle q \rangle \mod 2^n$ and $\kappa=1$ in all other cases.
	\end{itemize}
	\end{corollary}
	A complete classification of metacyclic 2-groups is available in \cite{XZ06} and hence  one can similarly do the computations for any metacyclic $2$-groups.

	\begin{remark}
		In \cite{CM25}, we determined the pcis of 
		$\mathbb{F}_q G$ for split metacyclic groups of order $p_1^m p_2^l$, 
		assuming $p_1$ and $p_2$ to be distinct odd primes. In view of results in this section, the case when $p_1=p_2=2$ is handled. Moreover, 
		using \Cref{lemma_trace_zero 2} and following the same reasoning as in the proofs of 
	\Cref{prop_Dihedral} and \Cref{prop_G_{2^{n+1}}},  the result of  \cite[Theorem~3.3]{CM25} 
	also extends to the case when one of the primes is $2$.
	\end{remark}

In the next section, we further extend our analysis to the remaining case when $p_1$ and $p_2$ are not  distinct but odd primes, that is, we consider metacyclic $p$-groups, where $p$ is an odd prime.

	\section{Metacyclic $p$-group codes, $p\neq 2$}
Cyclic and abelian $p$-group codes have been largely studied. For instance, the results of Arora and Pruthi \cite{PMA97} on cyclic $p$ codes were extended by Ferraz and Polcino Milies \cite{FRP07} to the results on abelian $p$-group codes.  In this section, we take a step further and study non-abelian metacyclic $p$-group codes, where $p$ is an odd prime.

Following Section 3, we first study metacyclic $p$-group codes for the groups with maximal cyclic subgroup. This is followed by assessing some arbitrary $p$-group codes.
\subsection{Metacyclic $p$-group codes having maximal cyclic subgroup, $p \neq 2$}
Up to isomorphism, there is a unique non-abelian metacyclic group of order $p^{n+1}$, where $p$ is an odd prime and $n \geq 2$, which possesses a maximal cyclic subgroup (c.f. (\cite{Huppert1}, I, Satz 14.9(a))). For this group, we provide a complete list of idempotents using strong Shoda pairs in the next theorem. 
\begin{theorem}\label{pcis_$p^n$}
For $G_{p^{n+1}}= \langle a, b \mid a^{p^{\,n}} = 1,\; b^p = 1,\; b^{-1}ab= a^{\,p^{\,n-1}+1} \rangle$, where $p$ is an odd prime and $n \geq 2$, the strong Shoda pairs of $G_{p^{n+1}}$ along with respective pcis of semisimple group algebra $\mathbb{F}_qG_{p^{n+1}}$ are as listed in \Cref{tab:pcis_G_pn}.
\begin{table}[ht]
	\centering
	\small
	\begin{tabular}{|p{4.2cm}|p{10.8cm}|}
		\hline
		\textbf{Strong Shoda pairs} & \textbf{Primitive central idempotents in $\mathbb{F}_q G_{p^{n+1}}$} \\
		\hline
		
		$(G_{p^{n+1}}, G_{p^{n+1}})$ & 
		$e_0 = \widehat{G_{p^{n+1}}}$ \\
		\hline
		
		$(G_{p^{n+1}}, \langle a \rangle)$ & 
		$e_{1,k} = \tfrac{1}{p}\,\widehat{\langle a \rangle}\;
		\sum\limits_{i=0}^{p-1} \tr(\xi_p^{\,ki})\,b^{-i}$ \\
		\hline
		
		$(G_{p^{n+1}}, \langle a^{p^j}, a^{ip^{j-1}} b \rangle)$, 
		
		$0 \leq i \leq p-1$, $1 \leq j \leq n-1$ 
		& 
		$e_{(i,j),k} =
		\begin{cases}
			\frac{1}{p^j}\,\widehat{\langle a^{p^j}, a^{ip^{j-1}} b \rangle}\;
			\sum\limits_{r=0}^{p^j-1} \tr(\xi_{p^j}^{\,kr})\,a^{-r}, & 1 \leq j \leq i_0, \\[6pt]
			\frac{1}{p^j}\,\widehat{\langle a^{p^j}, a^{ip^{j-1}} b \rangle}\;
			\sum\limits_{r=0}^{p^{i_0}-1} \tr(\xi_{p^j}^{\,kr})\,a^{-rp^{\,j-i_0}}, & j > i_0,
		\end{cases}$ \\
		\hline
		
		$(\langle a \rangle, \langle a^{p^n}\rangle)$ 
		& 
		$e_{p^n,k} =
		\begin{cases}
			\frac{1}{p^{i_0}} \sum\limits_{t \in T}\;\sum\limits_{r=0}^{p^{i_0}-1} 
			\tr(\xi_{p^{i_0}}^{\,krt})\,a^{-r}, & n = i_0, \\[6pt]
			\frac{1}{p^{\,n}} \sum\limits_{t \in T}\;\sum\limits_{r=0}^{p^{i_0}-1} 
			\tr(\xi_{p^{\,n}}^{\,kr t})\,a^{-rp^{\,n-i_0}}, & {\,n} > i_0,
		\end{cases}$ \\[6pt]
		& where $T$ is a transversal of $\langle (1+p^{n-1})^{\omega_0}\rangle$ in 
		$\langle 1+p^{n-1}\rangle$, with $\omega_0$ being the least integer such that 
		$(1+p^{n-1})^{\omega_0} \in \langle q \rangle$. \\
		\hline
		
	\end{tabular}
	\caption{Pcis of $\mathbb{F}_q G_{p^{n+1}}$}
	\label{tab:pcis_G_pn}
	\normalsize

\end{table}
\end{theorem}
The following corollaries are immediate from \Cref{pcis_$p^n$}
\begin{corollary}
	The structure of $	\mathbb{F}_qG$, where $G:=G_{p^{n+1}}$, is
\begin{equation*}\label{wedder p^n}
	\mathbb{F}_qG \cong
	\begin{cases}
		\mathbb{F}_q 
		\bigoplus \dfrac{\phi(p)}{o_p(q)}\,\mathbb{F}_{q^{o_p(q)}}
		\bigoplus  \limits_{j=1}^{n-1}p\dfrac{\phi(p^j)}{o_{p^j}(q)}\,\mathbb{F}_{q^{o_{p^{j}}(q)}} 
		\bigoplus \dfrac{\phi(p^{\,n})}{o_{p^{\,n}}(q)}  
		M_p\!\left( \frac{\mathbb{F}_{q^{o_{p^{\,n}}(q)}}}{p} \right), 
		 \mathrm{if}~ 1+p^{\,n-1} \in \langle q \rangle ~\mathrm{mod}~ p^{n}, \\
		\mathbb{F}_q 
		\bigoplus \dfrac{\phi(p)}{o_p(q)}\,\mathbb{F}_{q^{o_p(q)}}
		\bigoplus  \limits_{j=1}^{n-1}p\dfrac{\phi(p^j)}{o_{p^j}(q)}\,\mathbb{F}_{q^{o_{p^{j}}(q)}} 
		\bigoplus \dfrac{\phi(p^{\,n})}{p o_{p^{\,n}}(q)}  
		M_p\!\left( \mathbb{F}_{q^{o_{p^{\,n}}(q)}} \right),
		\text{otherwise.}
	\end{cases}
\end{equation*}
\end{corollary}

%
%

			\begin{corollary}\label{Cor_distance_{p^n}}In the foregoing notation, 
				\begin{enumerate}
						\item $\dim_{\mathbb{F}_q}(\mathbb{F}_qG_{p^{n+1}}e_{p^n, k})=
				po_{p^{n}}(q)\gcd(\omega_0, p)$, where  $\omega_0$ is the least integer such that 
				$(1+p^{n-1})^{\omega_0} \in \langle q \rangle$. 
						\item 	The set  
							$\mathcal{B}=
						\{a^{\eta_a}b^{\eta_b}e_{p^n,k}~|~0\leq \eta_a< \gcd(\omega_0,p)o_{p^{n}}(q),~0\leq \eta_b \leq p-1\}$, where  $\omega_0$ is the least integer such that 
						$(1+p^{n-1})^{\omega_0} \in \langle q \rangle$, is an $\mathbb{F}_q$-basis of $\mathbb{F}_qG_{p^{n+1}}e_{p^n, k}$.
						\item $d=\begin{cases}
						2 \leq d \leq p^{n} ~~\mathrm{if}~~ n\leq i_0\\
						2 \leq d \leq p^{i_0} ~~\mathrm{otherwise}.
						\end{cases}$

					\end{enumerate}   	
			\end{corollary}
	In view of \Cref{tab:pcis_G_pn} and  the Wedderburn decomposition of $\mathbb{F}_qG_{p^{n+1}}$, we have following:
	
	\begin{corollary}\label{left pcis F_qG_{p^{n+1}}}
	The semisimple group algebra $\mathbb{F}_qG_{p^{n+1}}$, with $G_{p^{n+1}}$ as defined in \Cref{pcis_$p^n$}, decomposes into minimal left ideals generated by a complete set of primitive orthogonal idempotents, given by:
	\begin{itemize}
			\item[(i)] $1 \;+\frac{\phi(p)}{o_p(q)}\,+ p\sum \limits_{j=1}^{n-1}\frac{\phi(p^j)}{o_{p^j}(q)}$ central  idempotents corresponding to strong Shoda pairs of type $(G_{p^{n+1}},K),$ as listed in \Cref*{tab:pcis_G_pn}.
			\item[(ii)] $2\frac{\phi(p^{n})}{\kappa o_{p^{n}}(q)}$ left  idempotents: $e_{p^{n},k}\widehat{\langle b \rangle}$ and $e_{p^{n},k}(1-\widehat{\langle b \rangle})$,	where $\kappa=p$ for $1+p^{n-1} \notin \langle q \rangle \mod p^{n}$ and $\kappa=1$ in all other cases.
		\end{itemize}
	\end{corollary}
%
Note that
the group 
	$G_{2^{n+1}} = \langle a, b \mid a^{2^{n}} = 1,\ b^2 = 1,\ b^{-1}ab = a^{1 + 2^{n-1}} \rangle$ ($G_{p^{n+1}}$ of \Cref{pcis_$p^n$} with $p = 2$), 
	 is the metacyclic $2$-group listed as (iv) in Subsection~3.2 and the results obtained in \Cref{pcis_$p^n$} coincide with those derived earlier.

	\subsection{Metacyclic $p$-group codes, $p\neq 2$}
		The classification of non-abelian  metacyclic $p$-groups of order $p^n, n \geq 3$ has been provided in \cite{Ste96}. One can, in principle use the presentation to compute the strong Shoda pairs and hence the idempotents for any of these groups. The strong Shoda pairs of all groups (not necessarily metacyclic) of order  $p^n, n \leq 4$ have been provided in  \cite{BM14} and \cite{BM15} using which the structure of their respective group algebras has been provided in \cite{GM19}. We list the strong Shoda pairs for metacyclic groups of order $p^5$ and consequently provide the description of their group algebras. 
		\begin{proposition}\label{prop_metacyclic_codes_p^5} For an odd prime $p$, there are four non isomorphic non-abelian metacyclic groups of order $p^5$ given by :
		
		\[
		\begin{aligned}
			G_1 &= \langle a, b \mid a^{p^2} = 1,\; b^{p^3} = 1,\; bab^{-1} = a^{p+1} \rangle,\\
			G_2 &= \langle a, b \mid a^{p^3} = 1,\; a^{p^2}=b^{p^2} ,\; bab^{-1} = a^{p+1} \rangle,\\
			G_3 &= \langle a, b \mid a^{p^3} = 1,\; a^{p^2}=b^{p^2} ,\; bab^{-1} = a^{p^2+1} \rangle,\\
			G_4 &= \langle a, b \mid a^{p^4} = 1,\; a^{p}=b^{p} ,\; bab^{-1} = a^{p^3+1} \rangle.
		\end{aligned}
		\]
		A complete set  $\mathcal{S}(G_i)$ of $G_i, 1 \leq i \leq 4$ is as listed below:
		
		\begin{small}
		
			\begin{description}
			\item[$\mathcal{S}(G_1)$] 	$
		=
			\Big\{ (G_1, \langle a, b^{p^i} \rangle) \Big\}_{i=0}^2\cup \Big\{ (G_1, \langle a^i b, a^p \rangle) \Big\}_{i=0}^{p-1}
			\cup \Big\{ (G_1, \langle ab^{ip}, a^p \rangle) \Big\}_{i=0}^{p-1}
			\cup \Big\{ (G_1, \langle ab^{-ip^2}, a^p \rangle) \Big\}_{i=1}^{p-1}
			\cup \Big\{ (\langle a^2 b, b^p \rangle,\; \langle a^{ip} b^{p^2} \rangle) \Big\}_{i=1}^{p-1}
			\cup \Big\{ (\langle a^2 b, b^p \rangle,\; \langle a^{ip} b^p, b^{p^2} \rangle) \Big\}_{\substack{0 \leq i \leq p-1 \\ i \neq 2}}
			\cup \Big\{ (\langle a^2 b, b^p \rangle,\; \langle a^{2p+2} b^{1-2p}, b^{p^2} \rangle) \Big\}.
$
		\item[$\mathcal{S}(G_2)$] 	$
	=
		\Big\{ (G_2, \langle a, b^{p^i} \rangle) \Big\}_{i=0}^{2}
		\cup \Big\{ (G_2, \langle a^k b, a^p \rangle) \Big\}_{k=0}^{p-1}
		\cup \Big\{ (G_2, \langle ab^{kp}, a^p \rangle) \Big\}_{k=1}^{p-1}
		\cup \Big\{ (\langle a^{-1}b, a^p \rangle,\; \langle a^{kp} b^p, b^{p^2} \rangle) \Big\}_{k=0}^{p-1}
		\cup \Big\{ (\langle a^{-1} b, b^p \rangle,\; \langle a^{-1} b^{1-2p}, b^{p^2} \rangle) \Big\}
		\cup \Big\{ (\langle ab^{p^2-p} \rangle, \langle 1 \rangle) \Big\}.$
		\item[$\mathcal{S}(G_3)$] =	$
		\Big\{ (G_3, \langle a^k b, b^{p^2} \rangle) \Big\}_{k=0}^{p^2-1}
		\cup \Big\{ (G_3, \langle a b^{kp}, b^{p^2} \rangle) \Big\}_{k=1}^{p-1}
		\cup \Big\{ (G_3, \langle a^k b, a^p \rangle) \Big\}_{k=0}^{p-1}
		\cup \Big\{ (G_3, \langle a, b^{p^i} \rangle) \Big\}_{i=0}^{2}
		\cup \linebreak\Big\{ (\langle a, b^{p} \rangle,\; \langle a^{p(p-1)} b^{p(pk+1)} \rangle) \Big\}_{k=0}^{p-1}.$
			\item[$\mathcal{S}(G_4)$] =	$
		\Big\{ (G_4, \langle a b^{-1}, b^{p^i} \rangle) \Big\}_{i=0}^{3}
		\cup \Big\{ (G_4, \langle b \rangle) \Big\}
		\cup \Big\{ (\langle a^p, a^{-1} b^2 \rangle,\; \langle 1 \rangle) \Big\}
		\cup \Big\{ (G_4, \langle a b^{k p^i - 1} \rangle) \Big\}_{\substack{1 \leq k \leq p-1 \\ 0 \leq i \leq 2}}.
		$
	\end{description}
\end{small}
\end{proposition}
As a direct consequence we obtain the structure of $\mathbb{F}_qG_i, 1 \leq i \leq 4$ and observe that no two of them are isomorphic.
\begin{enumerate}
	\item[$\mathbb{F}_qG_1 \cong$] $	
 \mathbb{F}_q  \bigoplus \delta  (p+1) \mathbb{F}_{q^{o_p(q)}} \bigoplus  \delta p \Big[
	\mathbb{F}_{q^{o_{p^2}(q)}} \bigoplus   \mathbb{F}_{q^{o_{p^3}(q)}} \Big]  \bigoplus M_{p}\!\left(\mathbb{F}_{q^{ o_{p}(q)}}\right)
	\bigoplus$ \\
	  $ \delta(p-1)\Big[ M_{p}\left(\mathbb{F}_{q^{ \frac{o_{p^3}(q)}{p}}}\right)  
	\bigoplus  M_{p}\!\left(\mathbb{F}_{q^{ \frac{o_{p^2}(q)}{p}}}\right)\Big] .
$
		\item[$\mathbb{F}_qG_2 \cong $] 	$ \mathbb{F}_q  \bigoplus \delta \Big[ (p+1) \mathbb{F}_{q^{o_p(q)}} \bigoplus  
		p\mathbb{F}_{q^{o_{p^2}(q)}}  \bigoplus  (p-1) M_{p}\!\left(\mathbb{F}_{q^{ \frac{o_{p^2}(q)}{p}}}\right) \Big] 
		\bigoplus \delta \Big[ M_{p}\!\left(\mathbb{F}_{q^{ \frac{o_{p^2}(q)}{p}}}\right)  \bigoplus  M_{p}\!\left(\mathbb{F}_{q^{ o_{p}(q)}}\right)\Big].$
			\item[$\mathbb{F}_q G_3 \cong$] 	$ \mathbb{F}_q  \bigoplus  \delta \Big[ (p+1) \mathbb{F}_{q^{o_p(q)}} \bigoplus  
			p(p+1)\mathbb{F}_{q^{o_{p^2}(q)}} \bigoplus  p M_{p}\!\left(\mathbb{F}_{q^{ \frac{o_{p^3}(q)}{p}}}\right) \Big].$
				\item[$\mathbb{F}_q G_4 \cong$] 	$  \mathbb{F}_q \bigoplus \delta \Big[ (p+1) \mathbb{F}_{q^{o_p(q)}} \bigoplus  
				p\mathbb{F}_{q^{o_{p^2}(q)}} \bigoplus p\mathbb{F}_{q^{o_{p^3}(q)}}\bigoplus    M_{p}\!\left(\mathbb{F}_{q^{ \frac{o_{p^4}(q)}{p}}}\right) \Big].$
		
\end{enumerate}
	
			By using  \Cref{prop_metacyclic_codes_p^5}, one can easily obtain the complete list of pcis of $\mathbb{F}_qG_i, 1 \leq i \leq 4.$ For instance, consider the group  $G_1$ and its strong Shoda pairs $\Big\{	(\langle a^2 b, b^p \rangle,\; \langle a^{ip} b^{p^2} \rangle)\Big\}_{i=1}^{p-1}$. Then, we have 
			$$\epsilon_C(\langle a^2 b, b^p \rangle,\; \langle a^{ip} b^{p^2} \rangle )=\frac{1}{p^3}\widehat{\langle a^{ip} b^{p^2} \rangle}[(\Sigma_{\mathfrak{i}=0}^{p-1}\tr(\xi_{p^3}^{kp^2\mathfrak{i}}){{(a^2b)}^{-p^2\mathfrak{i}}})]$$ and  since $(a^2b)^{-p^2\mathfrak{i}}
			\in \mathcal{Z}(G_1)$, it follows that $e_C(G_1, \langle a^2 b, b^p \rangle,\; \langle a^{ip} b^{p^2} \rangle )=\epsilon_C(\langle a^2 b, b^p \rangle,\; \langle a^{ip} b^{p^2} \rangle).$ Consequently, we obtain a minimal group code of length $p^5$ and dimension $p  o_{p^3}(q)$, with minimum distance satisfying 
			\[
			2p \;\leq\; d \;\leq\; p^2,
			\]
			for all admissible choices of $i$ and $k$.
		\subsection{Good $p$-group codes}
\begin{example}
	Over the field $\mathbb{F}_2$, for non-abelian  groups of order $27$, the pcis corresponding to $(G,K)$-types yield codes with parameters $[27, 2, 18]$ by using \Cref{parmameters $G=H$}, which coincide with the best-known binary linear codes of these parameters \cite{Gras}.
\end{example}
\begin{example}\label{example_dihedral_central}
	\underline{$\mathbb{F}_3G$, where	$G:=D_{8} = \langle a, b \mid a^{4} = b^2 = 1,\; a^b = a^{-1} \rangle$.}\\
	Consider the left idempotent $(1-e)\widehat{b}$ where,
	$e := e_C(G, G, G)$. Then the code generated by $(1-e)\widehat{b}$ has parameters $[8, 3, 4]$, which is very close to the best-known $[8, 3, 5]$ code.
\end{example}
We now also consider a non metacyclic $2$-group code.
\begin{example}\label{example_order16}
	\underline{$\mathbb{F}_3 G$, where $G \cong C_2 \times Q_8$}.\\
Let $G$ be presented as	\[
	G := \langle a, b, c \mid a^{4} = b^{4} = c^{2} = 1,\; ba = a^{3}b,\; ca = ac,\; cb = bc, \;a^{2} = b^{2} \rangle.
	\]
We have that
	\[
	\mathcal{S}(G) = \{(G, G),\; (G, \langle a, b\rangle),\; (G, \langle b^i c \rangle),\; (G, \langle a^2, a^i b, a^j c\rangle) (\langle a, c \rangle, \langle a^i c \rangle) \mid 0 \leq i, j \leq 1 \}.
	\]
	Consider
	$e=1-(e_1 + e_2+e_3),$ where $e_1=e_C(G, G, G)=\widehat{G}$, $e_2=e_C(G, G, \langle a, b \rangle)$ and $e_3=e_C(G,\langle a,c\rangle, \langle c \rangle).$ 
	The code generated by $e$ is a $[16, 10, 4]$ code, and is a best-known code.
\end{example}

\section{Metacyclic group codes of arbitrary length}
So far we have considered metacyclic codes of groups whose order is divisible by at most two primes.  In this section, we  investigate metacyclic codes of length divisible by more than two primes. The direct products of metacyclic groups of relatively prime order are also considered.

\begin{theorem}\label{thm:character_sum_vanish}
	Let $p_1, p_2$ be distinct odd primes with $p_1 < p_2$ such that $p_1 \nmid (p_2-1)$, and  let $q$ be a natural number relatively prime to both $p_1$ and $p_2$. For $m, l \in \mathbb{N}$, suppose 	\[
	o_{p_1^{m}}(q) = \frac{\phi(p_1^{m})}{\delta_1}
~\text{and}~
	o_{p_2^{\ell}}(q) = \frac{\phi(p_2^{\ell})}{\delta_2},
	\]
	where $\delta_1 = p_1^{\,i_0^{(1)}-1}\delta_1'$ and  $\delta_2 = p_2^{\,i_0^{(2)}-1}\delta_2'$ with $\gcd(\delta_1', p_1) = \gcd(\delta_2', p_2) = 1$. 
	 If $n = p_1^{m} p_2^{\ell}$, then 
	 for every integer  $k$ with $\gcd(k,n) = 1$ and for every $j_1, j_2$ such that $1 \leq j_1 \leq m$ and $1 \leq j_2 \leq l$,
	\[
	\sum_{i=0}^{o_{p_1^{j_1}p_2^{j_2}}(q)-1} \xi_{p_1^{j_1}p_2^{j_2}}^{k q^i} = 0 
	\quad \Longleftrightarrow \quad 
	j_1 > i_0^{(1)} ~~\text{or}~~\ j_2 > i_0^{(2)}.
	\]
\end{theorem}
\begin{proof}
	Set $J = \{ q^i \mid 0 \leq i < o_{p_1^{j_1}p_2^{j_2}}(q)\}, 
	$ so that 	$\sum_{i=0}^{o_{p_1^{j_1}p_2^{j_2}}(q)-1} \xi_{p_1^{j_1}p_2^{j_2}}^{\,k q^i}
	= \sum_{q^i \in J} \xi_{p_1^{j_1}p_2^{j_2}}^{\,k q^i}.
	$
	Denote $ \mathrm{lcm}\!\left(\frac{p_1-1}{\delta_1'}, \frac{p_2-1}{\delta_2'}\right)$ by $\lambda$ and using division algorithm write $i=\lambda u+v$ with $0 \leq v \leq \lambda$.\\
		\emph{Case (i).} Suppose $i_0^{(1)} < j_1 \leq m$ and $i_0^{(2)} < j_2 \leq \ell$.  
	In this case 
	$o_{p_1^{j_1} p_2^{j_2}}(q) 
	= p_1^{j_1 - i_0^{(1)}} p_2^{j_2 - i_0^{(2)}} \lambda
	$ and $0 \leq u < p_1^{j_1 - i_0^{(1)}} p_2^{j_2 - i_0^{(2)}}$. We thus have for $q^i \in J$, $q^i=(q^\lambda)^u q^v$
which modulo $p_1^{j_1}p_2^{j_2}$ equals $(1 + p_1^{i_0^{(1)}} p_2^{i_0^{(2)}} c)^u q^v $ and 
	$(1 + p_1^{i_0^{(1)}} p_2^{i_0^{(2)}} c)^u q^v 
	\equiv q^v + p_1^{i_0^{(1)}} p_2^{i_0^{(2)}} w\mod~p_1^{j_1}p_2^{j_2}$, where $0 \leq w < p_1^{j_1-i_0^{(1)}}p_2^{j_2-i_0^{(2)}} $. Thus $J$ can be described as
	\[
	J = \bigl\{\, q^v + p_1^{i_0^{(1)}}p_2^{i_0^{(2)}} w 
	\ \bigm|\ 
	0 \leq v < \lambda,\ 
	0 \leq w < p_1^{j_1-i_0^{(1)}}p_2^{j_2-i_0^{(2)}} \bigr\}.
	\]
		It follows that
	\[
	\sum_{q^i \in J}\xi_{p_1^{j_1}p_2^{j_2}}^{\,k q^i}
	= \sum_{v=0}^{\lambda-1}\xi_{p_1^{j_1}p_2^{j_2}}^{\,k q^v}
	\sum_{w=0}^{p_1^{j_1-i_0^{(1)}}p_2^{j_2-i_0^{(2)}}-1}
	\bigl(\xi_{p_1^{j_1}p_2^{j_2}}^{\,k p_1^{i_0^{(1)}}p_2^{i_0^{(2)}}}\bigr)^w.
	\]
	The inner geometric sum vanishes since the base is a nontrivial root of unity of order $p_1^{j_1-i_0^{(1)}}p_2^{j_2-i_0^{(2)}}$. Hence the entire sum equals zero.
	
	\medskip
		\emph{Case (ii).} Suppose $i_0^{(1)} < j_1 \leq m$ and $1 \leq j_2 \leq i_0^{(2)}$ (the argument is symmetric if $i_0^{(2)} < j_2 \leq \ell$ and $1 \leq j_1 \leq i_0^{(1)}$). 
		In this case $o_{p_1^{j_1} p_2^{j_2}}(q) = p_1^{j_1 - i_0^{(1)}} \lambda.$ Hence, for  $q^i \in J$, writing $i = \lambda u + v$ where $0 \leq u < p_1^{j_1-i_0^{(1)}}.$
Therefore, $q^i
	\equiv (1 + p_1^{i_0^{(1)}}p_2^{i_0^{(2)}}c)^u q^v
	\equiv q^v + p_1^{i_0^{(1)}}p_2^{i_0^{(2)}} w 
	\mod{p_1^{j_1}p_2^{j_2}},$
	for	some $c,w \in \mathbb{Z}$.
	Therefore
	\[
	J = \bigl\{\, q^v + w p_1^{i_0^{(1)}} p_2^{i_0^{(2)}}
	\ \bigm|\ 
 0 \leq v < \lambda ,	0 \leq w < p_1^{j_1-i_0^{(1)}}\bigr\}.
	\]
	It follows that
	\[
	\sum_{q^i \in J} \xi_{p_1^{j_1}p_2^{j_2}}^{\,k q^i}
	= \sum_{v=0}^{\lambda-1} \xi_{p_1^{j_1}p_2^{j_2}}^{\,k q^v}
	\sum_{w=0}^{p_1^{j_1-i_0^{(1)}}-1}
	\bigl(\xi_{p_1^{j_1}p_2^{j_2}}^{\,k p_1^{i_0^{(1)}}	p_2^{i_0^{(2)}}}\bigr)^w.
	\]
	The inner sum vanishes since $\xi_{p_1^{j_1}p_2^{j_2}}^{\,k p_1^{i_0^{(1)}}p_2^{i_0^{(2)}}}$ is a nontrivial root of unity of order $p_1^{j_1-i_0^{(1)}}$. Hence the whole sum equals zero. 
	
	\medskip
	
	\emph{Case (iii).} Suppose $1 \leq j_1 \leq i_0^{(1)}$ and $1 \leq j_2 \leq i_0^{(2)}$.  
	In this case $o_{p_1^{j_1}p_2^{j_2}}(q)$ is coprime to $p_1p_2$.  and the sum 
	$\sum \limits_{i=0}^{o_{p_1^{j_1}p_2^{j_2}}(q)-1} \xi_{p_1^{j_1}p_2^{j_2}}^{k q^i} \neq 0.$
	\end{proof}
In view of \Cref*{lemma_trace_zero 2}, similar proof can be extended if one of the prime is equal to $2$. We provide the results without details. 
\begin{proposition}\label{thm:char_sum_vanish_2p_refined}
	Let $p$ be an odd prime and let $q$ be a power of a prime with $\gcd(q,2p)=1$. Write $q=\pm1+2^{\,i_0}c,$ where $ c$ is odd and $ i_0\ge2.$
	For $ m,\ell\in\mathbb{N}$, let $n=2^{m}p^{\ell}$ . Suppose
	$o_{p^{\ell}}(q)=\frac{\varphi(p^{\ell})}{\delta_p}$ with $ 
	\delta_p=p^{\,i_0^{(p)}-1}\delta_p'$, where $\gcd(\delta_p',p)=1.	$ 
	
	If $1\le j_1\le m$, $1\le j_2\le\ell$, then for every $k\in\mathbb{Z}$ with $\gcd(k,n)=1$,
	\[
	\sum_{t=0}^{o_{2^{j_1}p^{j_2}}(q)-1}\xi_{2^{j_1}p^{j_2}}^{\,kq^t}=0
	\]
	holds exactly in the following cases:
	\begin{enumerate}
		\item $j_2>i_0^{(p)}$; or
		\item $j_1>i_0$; or
		\item $j_1=2$ and $q\equiv -1+2^{\,i_0}c$, $c$ odd .
	\end{enumerate}
	\end{proposition}
\begin{remark}
	\begin{itemize}
	\item[(i)] Let 	$D_{2n} = \langle a, b~ | ~a^n = b^2 = 1, \; b^{-1}ab = a^{-1} \rangle,$ of order $2n$, where $n \geq 3.$\\
		A complete list of strong Shoda pairs of $G:=D_{2n}$ is given by  	\[
		\mathcal{S}(G) :=
		\begin{cases}
			\{(G, G), \ (G, \langle a \rangle), \ (\langle a \rangle, \langle a^{v} \rangle) \mid v \neq 1, \ v \mid n \}, & \text{if $n$ is odd}, \\[6pt]
			\{(G, G), \ (G, \langle a \rangle), \ (G, \langle a^2, b \rangle), \ (G, \langle a^2, ab \rangle), \ (\langle a \rangle, \langle a^{v} \rangle) \mid v > 2, \ v \mid n\}, & \text{if $n$ is even}.
		\end{cases}
		\]
		
		Thus, using \Cref*{thm:character_sum_vanish} and \Cref{thm:char_sum_vanish_2p_refined}, one can extend the computations of~\cite[Section~4]{CM25} to the case where $n$ is the product of two distinct primes, which can further be extended to the case arbitrary $n$ and $q$ such that $n$ and $q$ are arbitrary. Thereby, extending \cite{DFPM09, GR22b, GR22a, GR23, Mar15}.
		
			\item[(ii)]Likewise for $Q_{4m} = \langle a, b \mid a^{2m} = 1, \; b^2 = a^{m},  b^{-1}ab = a^{-1} \rangle$, the generalized quaternion group, 	of order $4m$,  $m \geq 2$, similar results may be obtained using 
		a complete list of strong Shoda pairs of $G:=Q_{4m}$  given by  
		\[
		\mathcal{S}(G) :=
		\begin{cases}
			\{(G, G), \ (G, \langle a \rangle), \ (G, \langle a^2 \rangle), \ (\langle a \rangle, \langle a^{v} \rangle) \mid v > 2, \ v \mid 2m \}, & \text{if $m$ is odd}, \\[6pt]
			\{(G, G), \ (G, \langle a \rangle), \ (G, \langle a^2, b \rangle), \ (G, \langle a^2, ab \rangle), \ (\langle a \rangle, \langle a^{v} \rangle) \mid v > 2, \ v \mid 2m \}, & \text{if $m$ is even}.
		\end{cases}
		\]
		
	Consequently, this generalizes the study of quaternion 
	 codes of order $4m$ to arbitrary $m$ (see~\cite{GY21}).
	\end{itemize}
\end{remark}

\begin{example}\label{example_dihedral_central}
	\underline{$\mathbb{F}_5D_{12}, D_{12}= \langle a, b \mid a^{6}=b^2=1,\; a^b=a^{-1}\rangle.$} \\
Set $
	e :=e_C(D_{12}, D_{12}, \langle a^2, ab \rangle) \;+\; \widehat{b} e_C(D_{12}, \langle a \rangle, \langle 1 \rangle).
$
	Then the group code $\mathbb{F}_5 D_{12}\, e$ has parameters $[12,3,8]$, which is best known.
\end{example}
\begin{remark}
Let $G_1$ and $G_2$ be two groups with $\gcd(|G_1|, |G_2|)=1$. If $(H_1, K_1)$ and $(H_2, K_2)$ are respectively the strong Shoda pairs of $G_1$ and $G_2$, then one can  check that 
 $(H_1 \times H_2, K_1 \times K_2)$ is a strong Shoda pair of $G_1 \times G_2$. Consequentely, $\mathcal{S}(G_1 \times G_2)$ is obtainable from $\mathcal{S}(G_1)$ and $\mathcal{S}(G_2)$. This enables us to work with groups of order divisible by more than two primes.
\end{remark}

We illustrate the above discussion with an explicit example.
\begin{example}\label{example:G1xG2}
	Consider the groups
	$G_1 \;=\; \langle a_1,\,b_1 \mid a_1^{7^3}=1,\; b_1^{3}=1,\; b_1a_1b_1^{-1}=a_1^{18}\rangle$ and $G_2 \;=\; \langle a_2,\,b_2 \mid a_2^{11}=1,\; b_2^{5}=1,\; b_2a_2b_2^{-1}=a_2^{4
	}\rangle.$
		Here, $	
	\mathcal{S}(G_1)\;=\;\{(G_1,G_1),\ (G_1,\langle a_1\rangle),\ (\langle a_1\rangle,\langle a_1^{7^{j_1}}\rangle)_{j_1=1}^{3}\}$ and $	\mathcal{S}(G_2)\;=\;\{(G_2,G_2),\ (G_2,\langle a_2\rangle),\ (\langle a_2\rangle,\langle 1\rangle)\}.$
	Thus for $G=G_1 \times G_2$, the strong Shoda pairs of $G$ are listed in \Cref*{tab:pcis_G1xG2} along with the  expressions of pcis in $\mathbb{F}_2G$ computed  using \Cref*{thm:character_sum_vanish}. 
	\begin{center}
		\begin{table}[ht]
			\centering
			\begin{tabular}{|c|l|}
				\hline
				\textbf{Strong Shoda Pair $(H,K)$} & \textbf{Primitive central idempotent $e$ in $\F_2[G]$} \\
				\hline
				
				$(G,G)$ & \(\displaystyle e_0=\widehat{G}.\) \\[6pt]
				\hline
				
				$(G,\ \langle a_1\rangle\times G_2)$ & 
				\(
				\displaystyle e_{1,k} \;=\; \frac{1}{3}\,\widehat{\langle a_1\rangle\times G_2}\;
				\sum_{i=0}^{2}\tr\!\big(\xi_{3}^{k i}\big)\,(b_1,  1  )^{-i}
				\)
				
				\\
				\hline
				
				$(G,\ G_1\times\langle a_2\rangle)$ &
				\(
				\displaystyle e_{2,k} \;=\; \frac{1}{5}\,\widehat{G_1\times\langle a_2\rangle}\;
				\sum_{i=0}^{4}\tr\!\big(\xi_{5}^{k i}\big)\,(1, b_2)^{-i},
				\)
				\\[10pt]
				\hline
				
				$(G,\ \langle a_1\rangle\times\langle a_2\rangle)$ &
				\(
				\displaystyle e_{3,k}
				= \frac{1}{15}\,\widehat{\langle a_1\rangle\times\langle a_2\rangle}
				\sum_{i=0}^{14}
				\tr\!\big(\xi_{15}^{k i}\big)
				(b_1, b_2)^{-i},
				\)
				\\[12pt]
				\hline
				
				$(\langle a_1\rangle\times G_2,  \langle a_1^{7^{j_1}}\rangle)\times G_2 )_{j_1=1}^{3}$ &
				\(
				\displaystyle e_{4,k}^{j_1} \;= \frac{1}{7^{j_1}}\widehat{ \langle a_1^{7^{j_1}}\rangle)\times G_2}\sum_{i=0}^{6}\tr\!\big(\xi_{7^{j_1}}^{7^{j_1-1}ki}\big)\,(a_1, 1 )^{-i7^{j_1-1}}
				\)
				\\[8pt]
				\hline
				$(\langle a_1\rangle\times G_2,  \langle a_1^{7^{j_1}}\rangle)\times \langle a_2 \rangle )_{j_1=1}^{3}$ &
				\(
				\displaystyle e_{5,k}^{j_1} \;= \frac{1}{7^{j_1}}\widehat{\langle a_1^{7^{j_1}}\rangle)\times \langle a_2 \rangle}\sum_{i=0}^{6}\tr\!\big(\xi_{5\cdot7^{j_1}}^{7^{j_1-1}ki}\big)\,(a_1, b_2 )^{-i7^{j_1-1}}
				\)
				\\ \hline

				$(G_1\times\langle a_2\rangle, G_1\times \langle 1 \rangle)$ &
				\(
				\displaystyle e_{6,k} \;=\; \frac{1}{11}\widehat{G_1}\sum_{i=0}^{10}\tr\!\big(\xi_{11}^{ki}\big)(a_2)^{-i},
				\)
				\\[8pt]
				\hline
				$(G_1\times\langle a_2\rangle, \langle a_1 \rangle \times \langle 1 \rangle )$ &
				\(
				\displaystyle e_{7,k} \;=\; \frac{1}{33}\sum_{i=0}^{32}\tr\!\big(\xi_{33}^{ki}\big)(b_1, a_2)^{-i},
				\)
				\\[8pt]
				\hline
				
				$(\langle a_1\rangle\times\langle a_2\rangle, \langle a_1^{7^{j_1}}\rangle \times \langle 1 \rangle)_{j_1=1}^{3}$ &
				\(
				\displaystyle e_{8,k}^{j_1} \;= \frac{1}{7^{j_1}\cdot 11}\widehat{\langle a_1^{7^{j_1}}\rangle}
				\sum_{i=0}^{6}
				\tr\!\big(\xi_{7^{j_1}\cdot 11}^{i7^{j_1-1}k}\big)(a_1, a_2)^{-i7^{j_1-1}}.
				\)
				\\[12pt]
				\hline
					\end{tabular}
			\caption{Pcis of $\mathbb{F}_2(G_1 \times G_2)$.}
			\label{tab:pcis_G1xG2}
		\end{table}
			\end{center}
	Further, the codes generated by the pcis  corresponding to strong Shoda pairs $(H, K)$, where $H$ is proper subgroup of $G$ are provided in \Cref*{tab:params_G1xG2_noncentral}.
	\begin{table}[ht]
		\centering
		\setlength{\tabcolsep}{8pt}
		\renewcommand{\arraystretch}{1.2}
		\begin{tabular}{|l|c|c|c|}
			\hline
			\textbf{Strong Shoda pair $(H,K)$} & \textbf{Idempotent} & \textbf{Dimension $k$} & \textbf{Distance $d$} \\ \hline
			$(\langle a_1\rangle\times G_2,\ \langle a_1^{7^{j_1}}\rangle\times G_2)_{j_1=1}^{3}$ 
			& $e_{4,k}^{j_1}$ & $ 9\cdot 7^{j_1-1}$ & $110 \cdot 7^{3-j_1}  \leq d \leq 330 \cdot 7^{3-j_1}$\\ \hline
			$(\langle a_1\rangle\times G_2,\ \langle a_1^{7^{j_1}}\rangle\times\langle a_2\rangle)_{j_1=1}^{3}$ 
			& $e_{5, k}^{j_1}$ & $36\cdot 7^{j_1-1}$ & $22 \cdot 7^{3-j_1} \leq d \leq 66\cdot 7^{3-j_1}$ \\ \hline
			$(G_1\times\langle a_2\rangle,\ G_1\times \langle 1\rangle)$ 
			& $e_{6, k}$ & $50$ & $6\cdot 7^3 \leq d \leq 10\cdot 7^3$ \\ \hline
			$(G_1\times\langle a_2\rangle,\ \langle a_1\rangle\times \langle 1\rangle)$ 
			& $e_{7, k}$ & $50$ & $2\cdot 7^3 \leq d \leq 33\cdot 7^3$ \\ \hline
			$(\langle a_1\rangle\times\langle a_2\rangle,\ \langle a_1^{7^{j_2}}\rangle\times \langle 1\rangle)_{j_1=1}^{3}$ 
			& $e_{8, k}^{j_1}$ & $450\cdot 7^{j_1-1}$ & $2\cdot 7^{3-j_1} \leq d \leq 6\cdot 7^{3-j_1}$ \\ \hline
		\end{tabular}
		\caption{Parameters for codes corresponding to pcis of $\mathbb{F}_2(G_1 \times G_2).$}
		\label{tab:params_G1xG2_noncentral}
	\end{table}
	
\end{example}
\section{Some Non-Central Good Codes}
In this section, we aim to construct non-central codes arising from the central ones.  As observed in \cite{CM25}, non-central codes often exhibit better distance parameters and they are not always equivalent to abelian codes. Since such codes correspond to left ideals, we need to consider left idempotents. In view of \Cref{left pcis F_qD_{2^{n+1}}}, \Cref{left pcis F_qG_{2^{n+1}}}, and \Cref{left pcis F_qG_{p^{n+1}}}, it is sufficient to consider left group codes generated by idempotents of the form 
$e_C(G, H, K)\widehat{\langle b \rangle}$,
where \( H \) is a proper subgroup of the metacyclic group \( G \). Since the construction of these codes essentially depends on explicit expressions of the idempotents, we perform the computations for the idempotents \( e_{p_1^{j_1}, k_1}\in \mathbb{F}_qG \), which were obtained in our earlier work~\cite{CM25} for $G$ as in (\ref{equation 3}). We first consider codes generated by 
$\mathbb{F}_q G e_{p_1^{j_1}, k_1} \widehat{\langle b^\beta \rangle}, \quad 1 \leq \beta \leq p_2^l - 1.$
\begin{theorem}\label{thm:noncentral-code}
	Let $C=\mathbb{F}_q G e_{p_1^{j_1}, k_1}\widehat{\langle b^\beta \rangle}$ be a non-central code.  
	If $\dim$ and $d$ respectively denote its dimension and minimum distance, then
	\[
	\dim = o_{p_1^{j_1}}(q)\, p_2^{\lambda+\lambda_0},
	\]
	and
	\[
	d=
	\begin{cases}
		2 p_1^{m-j_1}\, p_2^{\,l-\lambda} \;\leq d \leq\; p_1^m \, p_2^{\,l-\lambda}, & 1 \leq j_1 \leq i_0^{(1)}, \\[4pt]
		2 p_1^{m-j_1} p_2^{\,l-\lambda} \;\leq d \leq\; p_1^{\,m-j_1+i_0^{(1)}} \, p_2^{\,l-\lambda}, & \text{otherwise},
	\end{cases}
	\]
	where $\gcd(\beta, p_2^l)=p_2^\lambda$ and $\gcd(\omega_0, p_2^l)=p_2^{\lambda_0}$.
\end{theorem}

\begin{proof}
	Let $f := e_{p_1^{j_1}, k_1}\widehat{\langle b^\beta \rangle}$ and $o := o_{p_1^{j_1}}(q)$.  
	
	Case 1: $r \in \langle q \rangle$.
	We claim that
	\[
	\mathcal{B} := \{\, a^{\eta_a} b^{\eta_b} f \;\mid\; 0 \leq \eta_a < o,\; 0 \leq \eta_b < p_2^\lambda \,\}
	\]
	is a basis for the left ideal $\mathbb{F}_q G f$.  	Indeed, for any $\alpha \in \mathbb{F}_q G f$, write $\alpha = \alpha' \widehat{\langle b^\beta \rangle}$ with  
	\[
	\alpha' = \sum_{\eta_b=0}^{p_2^l-1} \sum_{\eta_a=0}^{o-1} 
	\alpha_{\eta_a\eta_b}\, a^{\eta_a} b^{\eta_b} e_{p_1^{j_1}, k_1}.
	\]
	Partitioning $\eta_b$ by a transversal $T$ of $\langle b^\beta \rangle$ in $\langle b\rangle$ shows that $\mathcal{B}$ spans $\mathbb{F}_q G f$.  
	Linear independence follows from the independence of $\{a^{\eta_a} e_{p_1^{j_1},k_1}\}_{0\leq \eta_a<o}$.  
	Thus $|\mathcal{B}| = o\, p_2^\lambda$, giving $\dim = o\, p_2^\lambda$.  
	
	Case 2: $r \notin \langle q \rangle$. 
	By the same argument as in \cite[Theorem~3.5]{CM25}, the $\mathbb{F}_q$ basis of $\mathbb{F}_qGe_{p_1^{j_1},k_1}$ is
	\[
	\mathcal{B}' = \{\, a^{\eta_a} b^{\eta_b} \, \epsilon_C(\langle a \rangle,\langle a^{p_1^{j_1}}\rangle)^t \widehat{\langle b^\beta \rangle} 
	\;\mid\; 0\leq \eta_a<o,\; 0\leq \eta_b < p_2^\lambda,\; t\in \tau \,\},
	\]
	where $\tau$ is a transversal of $\langle a, b^{\omega_0}\rangle$ in $G$.  
	Hence $\dim = o\, p_2^{\lambda+\lambda_0}$.
	
For the distance bound, let $K = \langle a^{p_1^{j_1}} \rangle$ and $N = \langle b^\beta \rangle$, so that $NK \leq G$. By the method of \Cref{parmameters $G=H$}, the code parameters satisfy the stated bounds.

\end{proof}
\subsection*{Non-central codes using units}
We first recall some well known units of integral group ring  $\mathbb{Z}G$ which remain units in semisimple group ring  $\mathbb{F}_q G$  (for instance see \cite{DJ15}, 1.2.4).
\medskip

\noindent
\textbf{Bicyclic units.} For $g,h \in G$, with $\tilde{h}=1+h+\cdots+h^{|h|-1}$, define
\[
b(g, \tilde{h}) = 1 + (1 - h)g\tilde{h}, 
\quad 
b(\tilde{h}, g) = 1 + \tilde{h}g(1 - h).
\]
These are units with inverses
$b(g, \tilde{h})^{-1} = b(-g, \tilde{h})$ and $b(\tilde{h}, g)^{-1} = b(\tilde{h}, -g)$.

\medskip

\noindent 
\textbf{Bass units.} Consider $x \in G$, an element of order $n$. Let $k$ be relatively coprime to  $ n$ and let $m>0$ be such that $k^m \equiv 1 \mod{n}$. Then
\[
u_{k,m}(x) = \Big(1 + x + \cdots + x^{k-1}\Big)^m + \frac{1-k^m}{n}\,\tilde{x}
\]
is invertible, with inverse $u_{l,m}(x^k)$ where $kl \equiv 1 \mod{n}$.\\
\textbf{Alternating units.} Let $g \in G$, an element of odd order $n$ and let $k$ coprime to $2n$. Then,$$u_k(g) := 1 + g + g^2 + \cdots + g^{k-1}$$ is a  unit.  Consider $u_{k}(g)$ in  $\mathbb{F}_q G$ where $q$ is a power of $2$
and  $g \in G$ be an element of order  $p^m$, $p\neq 2$. If $k <p$ and let $k_1 \in \mathbb{Z}_{>0}$ is such that 
\[
k k_1 \equiv 1 \mod{p^m},
\] then
\[
u_k(g)^{-1} =
\begin{cases}
	u_{k_1}(g^k), & \text{if } k_1 \text{ is odd}, \\
	u_{k_1}(g^k) + \tilde{g}, & \text{if } k_1 \text{ is even}.
\end{cases}
\] This is because
$u_k(g) \cdot u_{k_1}(g^k) = \left( \sum_{i=0}^{k-1} g^i \right) \left( \sum_{j=0}^{k_1 - 1} g^{kj} \right) = \sum_{t=0}^{k k_1 - 1} g^t.$
Since \( k k_1 \equiv 1 \mod{p^m} \) and \( g \) has order \( p^m \). Thus, the sum above is over \( k k_1 \equiv 1 \mod{p^m} \) consecutive powers of \( g \), which implies
\[
u_k(g) \cdot u_{k_1}(g^k) =
\begin{cases}
	1, & \text{if } k_1 \text{ is odd}, \\
	1 + \tilde{g}, & \text{if } k_1 \text{ is even}.
\end{cases}
\]
Reconsidering $G$ as in (\ref{equation 3}) and its idempotents $e:=e_{p_1^{j_1}, k_1}$ as  given in (\cite{CM25}, Theorem 3.3), we have that, if $u \in \mathbb{F}_q G e$ is a unit, then 
$ue\widehat{\langle b^\beta \rangle}u^{-1}
$ yields non-central idempotents since conjugation is an automorphism. It has been observed that suitable choices of units can improve the parameters of such codes \cite{APM19}. Futher, 
$e + s\,\widehat{\langle b \rangle}a^k (1 - \widehat{\langle b \rangle})e$ is a unit with inverse
$e - s\,\widehat{\langle b \rangle}\,a^k (1 - \widehat{\langle b \rangle})\, e$
where,  $s \in \mathbb{F}_q \setminus \{0\}$ and  $1 \leq k \leq o(a) - 1$, $o(a)$ being the order of $a$.\\
For any unit $u$ defined above,
we denote $u^{-1}e\widehat{\langle b^\beta \rangle}  u$ by $ e^{\beta u}$ and $e \widehat{\langle b^\beta \rangle} $ by $ e^{\beta}$. In the ongoing notation we have the following corollary to \Cref{thm:noncentral-code}.
\begin{corollary}The dimension and distance of $\mathbb{F}_qGe^{\beta u}$ satisfy the following:
	\begin{enumerate}
	\item $\dim_{\mathbb{F}_q}(\mathbb{F}_qG e^{\beta u})=o_{p_1^{j_1}}(q) \cdot p_2^{\lambda + \lambda_0}
	$
	\item  
	$d(\mathbb{F}_qGe^{\beta }) \leq d(\mathbb{F}_qGe^{\beta u}).$
	
\end{enumerate}   	
\end{corollary}
%
\begin{proof}The dimension follows from \Cref{thm:noncentral-code}. Since conjugation is an automorphism the dimensions of code generated by the non-central idempotents constructed by units are the same as code generated by idempotents  $ e^{\beta}$. 
	Also if we consider $u=e+ s\widehat{\langle b^\beta \rangle}a^k(1-\widehat{\langle b^\beta \rangle})e$, then clearly
	$u^{-1}eu=(1-a(1-\widehat{ \langle b^\beta \rangle}))e\widehat{\langle b^\beta \rangle}.$
	Now if we consider $u=u_{k, m}(a)$ or $u=u_{k}(a)$ then  $u_{k, m}(a)e\widehat{\langle b^\beta \rangle}=e\widehat{\langle b^\beta \rangle}u_{k, m}(a^r),\; r \in \mathbb{Z}$, which implies idempotent $ e^{\beta u}$  contains $ e^{\beta}$ as a factor.  So, $d(\mathbb{F}_qGe^{\beta}) \leq d( e^{\beta u})$.
\end{proof}
Now we will give some examples where code generated by above methods  indeed improve the distance parameter.
\begin{example}\label{example_dihedral_central}\underline{$\mathbb{F}_3D_{14}$ and $\mathbb{F}_5D_{14}$, $D_{14}= \langle a, b ~|~ a^{7}=b^2=1, a^b=a^{-1} \rangle.$}\\
		We see that the pcis of $\mathbb{F}_3D_{14}$ are $e_1:=\widehat{D_{14}}$, $e_2:=\widehat{\langle a \rangle}-\widehat{D_{14}}$ and $e_{7,1}:=1-\widehat{\langle a \rangle}$. 
	
	If we consider  $\mathbb{F}_3D_{14}$, then the code generated by the idempotent $\widehat{\langle b \rangle}e_{7,1}$ is a $[10,6,4]$. However, by considering the idempotent $e_{7,1}(\widehat{\langle b \rangle}+\widehat{\langle b \rangle}a(1-\widehat{\langle b \rangle}))$, we get a $[14, 6, 6]$ code, which is best known and also  not equivalent to abelian code.\\
	Now if we consider  $\mathbb{F}_5D_{14}$, then code generated by  the idempotent $e_{7,1}(\widehat{\langle b \rangle}+\widehat{\langle b \rangle}a(1-\widehat{\langle b \rangle}))$, we get a $[14, 6, 7]$ code, which is best known and also  not equivalent to abelian code.\\ 
\end{example}
\begin{example}\label{example_order39}\underline{$\mathbb{F}_2G$ and $\mathbb{F}_5G$, $G$-the non-abelian metacyclic group of order $39$}, \\
	
	$G:= \langle a, b ~|~ a^{13}=b^3=1, a^b=a^{9} \rangle$.\\
	
	The code generated by the pci of $\mathbb{F}_2G$ corresponding to the strong Shoda pair $(\langle a \rangle,  \langle 1\rangle) $, namely $e_{13,1}=e_C(G, \langle a \rangle,  \langle 1\rangle)$ yields a $[39, 36, 2]$, which is a best known code.\\
	Now the non-central idempotent, namely $e_{13,1}\widehat{\langle b \rangle}$, gives a $[39, 12, 6]$ code.
	However, by following the technique given above and considering the idempotent $e_{13,1}(\widehat{\langle b \rangle}+\widehat{\langle b \rangle}a(1-\widehat{\langle b \rangle}))$, we get a $[39, 12, 10]$ code, which is not equivalent to abelian code. Now if we consider $u=1+a+a^2$ as proved above, $ue_{13, 1} $ is a unit in $\mathbb{F}_2Ge_{13, 1}$. So $e_{13, 1}^{1u}$  gives a code with parameters $[39, 12, 12]$.  \\
Similarly, if we consider the code generated by the non-central idempotent 
$e_{13,1}\widehat{\langle b \rangle}$ of $\mathbb{F}_5G$, we obtain a $[39,12,6]$-code. 
However, when we take $e_{13,1}(\widehat{\langle b \rangle} + \widehat{\langle b \rangle}a(1 - \widehat{\langle b \rangle}))$, 
the resulting code has parameters $[39,12,17]$, which is very close to the 
best-known $[39,12,18]$-code. Thus, the distance parameter increases 
significantly.

\end{example}

\begin{example}\label{example_order57}
	\underline{$\mathbb{F}_2G$, where $G$ is the non-abelian metacyclic group of order $57$},\\
	\[
	G := \langle a, b \mid a^{19} = b^3 = 1,~ a^b = a^{7} \rangle.
	\]
	The non-central idempotent $e_{19,1}\widehat{\langle b \rangle}$ gives a $[57, 18, 6]$ code. However, by following the technique described above and considering the idempotent
$
	e_{19,1}(\widehat{\langle b \rangle} + \widehat{\langle b \rangle}a(1 - \widehat{\langle b \rangle})),
$
	we obtain a $[57, 18, 14]$ code, which is not equivalent to any abelian code.
	
	Now, if we consider the unit $u = 1 + a + a^2$, then $u e_{19,1}$ is a unit in $\mathbb{F}_2 G e_{19,1}$. So the conjugated idempotent $e_{19,1}^{1u}$ gives a code with parameters $[57, 18, 16]$. Again, this code is inequivalent to any abelian code and is close to the best-known code with parameters $[57, 18, 17]$. It is apparent that the distance parameter increases. 
\end{example}
\begin{example}\label{example_order20}
	\underline{$\mathbb{F}_3G$, where $G$ is the non-abelian metacyclic group of order $20$},\\
	\[
	G := \langle a, b \mid a^{5} = b^4 = 1,~ a^b = a^{2} \rangle.
	\]
	The non-central idempotent $e_{5,1}\widehat{\langle b \rangle}$ generates a code with parameters $[20,4,8]$.  
	Now consider the unit $u = 1 + a $. Since $u e_{5,1}$ is a unit in $\mathbb{F}_3 G e_{5,1}$, the conjugated idempotent $e_{5,1}^{1u}$ yields a code with parameters $[20,4,12]$. 
	This code coincides with a best-known code. Clearly, adjoining the unit improves the minimum distance from $8$ to $12$ while keeping the length and dimension fixed.
\end{example}

			\bibliographystyle{amsalpha}
			\bibliography{Codes}

		\end{document}